\documentclass[12pt,twoside,reqno]{amsart}

\usepackage{amssymb,amsmath,amstext,amsthm,amsfonts,amscd,xcolor}

\usepackage{dsfont}

\usepackage[ansinew]{inputenc} 
\usepackage{graphicx}
\usepackage[mathscr]{eucal}
\usepackage{hyperref}

\newcommand{\R}{\mathbb{R}}
\newcommand{\C}{\mathbb{C}}
\newcommand{\N}{\mathbb{N}}

\newcommand{\GL}{{\rm GL}}
\newcommand{\gl}{{\rm Mat}}
\newcommand{\Mat}{{\rm Mat}}

\newcommand{\Bundle}{\mathcal{B}}
\newcommand{\deltamin}{\delta_{{\rm min}}}

\newcommand{\avg}[1]{\left< #1 \right>} 
\newcommand{\abs}[1]{\bigl| #1 \bigr|} 
\newcommand{\norm}[1]{\lVert#1\rVert} 
\newcommand{\bnorm}[1]{\Bigl\| #1\Bigr\|} 

\newcommand{\transl}{{T}} 
\newcommand{\less}{\lesssim}
\newcommand{\more}{\gtrsim}
\newcommand{\ep}{\epsilon} 
\newcommand{\ka}{\kappa}

\newcommand{\cocycles}{\mathcal{C}}
\newcommand{\observables}{\Xi}
\newcommand{\param}{\underline{p}}
\newcommand{\params}{\mathscr{P}}

\newcommand{\gapc}{\kappa}
\newcommand{\nzerobar}{\underline{n_0}}

\newcommand{\mesfs}{\mathscr{I}}

\newcommand{\dev}{\ep}                
\newcommand{\devf}{\underline{\dev}}  
\newcommand{\mes}{\iota}              
\newcommand{\mesf}{\underline{\mes}}  

\newcommand{\scale}{\mathscr{N}}  

\newcommand{\An}[1]{A^{({#1})}}  
\newcommand{\Bn}[1]{B^{({#1})}}  

\newcommand{\Lan}[2]{L^{({#1})}_{#2}}

\newcommand{\Gr}{{\rm Gr}}
\newcommand{\Pp}{\mathbb{P}}
\newcommand{\FF}{\mathscr{F}}

\newcommand{\hatv}{\hat{v}}

\newcommand{\U}{\mathscr{U}}
                
\newcommand{\ind}{\mathds{1}}

\newcommand{\B}{\mathscr{B}}

\newcommand{\dist}{{\rm dist}}

\newcommand{\Proj}{\mathbb{P}(\R^{m})}

\newcommand{\filt}{F}

\newcommand{\Decompsp}{\mathscr{D}}

\newcommand{\mostexp}{\overline{\mathfrak{v}}}

\newcommand{\aangle}{\alpha}

\newcommand{\sgap}{\sigma}
\newcommand{\rgap}{{\rm gr}}
\newcommand{\rift}{\rho}

\newcommand{\Ker}{\rm K}
\newcommand{\Range}{\rm R}

\newcommand{\Filt}{\mathfrak{F}}

\newcommand{\Filatp}[1]{\mathfrak{F}_{\supset{#1}}}

\newcommand{\dec}{E_{\cdot}}
\newcommand{\Dec}{\mathfrak{D}}

\newcommand{\Decatp}[1]{\mathfrak{D}_{\supset{#1}}}

\newcommand{\medir}[1]{\mostexp^{({#1})}}

\makeatletter
\newsavebox{\@brx}
\newcommand{\llangle}[1][]{\savebox{\@brx}{\(\m@th{#1\langle}\)}%
  \mathopen{\copy\@brx\mkern2mu\kern-0.9\wd\@brx\usebox{\@brx}}}
\newcommand{\rrangle}[1][]{\savebox{\@brx}{\(\m@th{#1\rangle}\)}%
  \mathclose{\copy\@brx\mkern2mu\kern-0.9\wd\@brx\usebox{\@brx}}}
\makeatother

\newcommand{\normtwo}[1]{
{\left\vert\kern-0.25ex\left\vert\kern-0.25ex\left\vert #1 
    \right\vert\kern-0.25ex\right\vert\kern-0.25ex\right\vert} } 

\newcommand{\drel}{d_{{\rm rel}}}

\theoremstyle{plain}
\newtheorem{theorem}{Theorem}[section]
\newtheorem{proposition}{Proposition}[section]
\newtheorem{corollary}[proposition]{Corollary}
\newtheorem{lemma}[proposition]{Lemma}
\newtheorem{definition}{Definition}[section]

\numberwithin{equation}{section}

\newtheorem{remark}{Remark}[section]

\title[Continuity of the Oseledets Decomposition]{Continuity of the Oseledets Decomposition}

\date{}

\begin{document}

\author[P. Duarte]{Pedro Duarte}
\address{Departamento de Matem\'atica and CMAFIO\\
Faculdade de Ci\^encias\\
Universidade de Lisboa\\
Portugal 
}
\email{pmduarte@fc.ul.pt}

\author[S. Klein]{Silvius Klein}
\address{ Department of Mathematical Sciences\\
Norwegian University of Science and Technology (NTNU)\\
Trondheim, Norway\\ 
and IMAR, Bucharest, Romania }
\email{silvius.klein@math.ntnu.no}

\begin{abstract} We consider an abstract space of measurable linear cocycles and we assume the availability in this space of some appropriate uniform large deviation type estimates. Under these hypotheses we establish the continuity of the Oseledets filtration and decomposition as functions of the cocycle.  The same assumptions lead in~\cite{LEbook} 
to a general continuity theorem for the Lyapunov exponents. This result and other technical estimates derived in~\cite{LEbook}, 
along with the inductive scheme based on the Avalanche Principle, are the main ingredients of the arguments in this paper.

We also  give a new proof of the classical Multiplicative Ergodic Theorem of V. Oseledets, using the Avalanche Principle (AP).

This is a draft of a chapter  in our forthcoming research monograph~\cite{LEbook}.  
\end{abstract}

\maketitle

\section{Introduction and statements}
\label{oseledets_introduction}
Let $(X, \mu, T)$ be an ergodic dynamical system and let $A \colon X \to \Mat (m, \R)$ be a measurable function defining a linear cocycle on the bundle space $X \times \R^m$ by
$$X \times \R^m \ni (x, v) \mapsto (T x, A (x) v) \in X \times \R^m.$$

In his 1968 paper \cite{Oseledets} in the Transactions of the Moscow Mathematical Society,  V. Oseledets proved his now famous Multiplicative Ergodic Theorem. Assuming the integrability of the cocycle,  this theorem proves the existence of a measurable and $(T, A)$-invariant filtration of the fiber 
$$\{0\}= F_{k+1}(x) \subsetneq F_k(x)
\subsetneq \ldots \subsetneq F_{2}(x) \subsetneq F_1(x) = \R^m,$$
and the existence of a sequence $\lambda_1>\lambda_2>\ldots >\lambda_k\geq -\infty$, \ such that  for $\mu$-a.e. phase $x \in X$ and 
for every vector $v\in F_j(x)\setminus F_{j+1}(x)$,
$$\lim_{n\to +\infty} \frac{1}{n}\,\log \norm{\An{n}(x)\,v} = \lambda_{j}\,.$$

The numbers $\lambda_1, \lambda_2, \ldots, \lambda_k$, measuring the rate of expansion of the cocycle along the invariant Oseledets subspaces, are the {\em distinct} Lyapunov exponents (LE). 

The {\em repeated} Lyapunov exponents  $L_1 (A) \ge L_2 (A) \ge \ldots \ge L_m (A)$ are defined by the Furstenberg-Kesten (or Kingman's sub-additive ergodic) theorem.  

Making further assumptions (e.g. the base dynamic and the fiber action are invertible), there is a measurable and $(T, A)$-invariant  decomposition (also called splitting) into subspaces
$\R^m =\oplus_{j=1}^{k+1} E_i(x)$,
such that for $\mu$-a.e. $x \in X$ and for every $v\in E_j(x)\setminus \{0\}$ we have
$ \lim_{n\to \pm \infty} \frac{1}{n}\,\log \norm{\An{n}(x)\,v} =\lambda_{j}$.

There are several methods of proving the multiplicative ergodic theorem. We mention the proofs of I. Ya. Gol'dsheid and G. Margulis in \cite{Goldsheid-Margulis-met} (for a detailed presentation of this proof see \cite{Arnold}),  R. Ma\~n\'e (see his monograph  \cite{Mane}),  P. Walters (see \cite{Walters-met})  as well as variants of these proofs by M. Viana (see his recent monograph \cite{Viana-book}) or J. Bochi (see the lecture notes \cite{Bochi-met} on his web page). 

Many  extensions of this theorem are available, including those of I. Ya. Gol'dsheid and G. Margulis in \cite{Goldsheid-Margulis-met} or D. Ruelle in \cite{Ruelle1, Ruelle2}.

\medskip

In this paper we give a new proof of the multiplicative ergodic theorem, which is based upon  the AP. More precisely, we use the estimate in the AP on the distance between the most expanding direction of a product of matrices and the most expanding direction of the first term in the product. 

We assume the base dynamics to be invertible. However, the existence of the Oseledets filtration for non-invertible base dynamics can be reduced to the invertible case by a natural extension construction (see Section 1.3 in \cite{Petersen}).

The Oseledets decomposition is usually obtained under the assumption that both the base dynamics and the fiber action are invertible. Our pooof does not require invertibility of the fiber action.

\medskip

We construct the Oseledets filtration as the $\mu$-a.e. limit as $n \to \infty$ of filtrations corresponding to the  singular value decomposition of the iterates $\An{n} (x)$ of the cocycle $A$. The convergence of these (finite scale) filtrations follows from our extension of the AP (see~\cite{LEbook, LEbook-chap2}) concerning estimates on the distance between most expanding singular directions of products of matrices. The assumptions of the AP are ensured by Kingman's ergodic theorem, which provides $\mu$-a.e. convergence to the Lyapunov exponents of certain quantities related to the iterates $\An{n} (x)$ of the cocycle.

\medskip

If a {\em quantitative} version of the convergence in Kingman's ergodic theorem is available, that is, if our system satisfies fiber large deviation type (LDT) estimates, then we establish a {\em rate of convergence} of the finite scale filtrations to the Oseledets filtration. 

Moreover, if the LDT is {\em uniform} in the cocycle, we derive {\em continuity} of the Oseledets filtration as a function of the cocycle, in an appropriate average sense. The argument is again inductive and based upon the AP, whose assumptions are shown to hold off of small sets of phases related to the exceptional sets in the LDT estimates. 

\medskip

We construct the subspaces of the Oseledets {\em decomposition} of the cocycle $A$ as intersections between components of the orthogonal complements of the filtration of $A$ and components of the filtration of the adjoint cocycle. 

The continuity of the Oseledets decomposition  (under the same assumption of having uniform LDT estimates) is derived using a similar scheme as the one employed for the continuity of the filtration. However, this needs to be combined with a careful analysis of the Lipschitz behavior of the intersection of vector subspaces, which we obtained in Chapter 2 of~\cite{LEbook} (see also our preprint~\cite{LEbook-chap2}). 

\medskip

A precise formulation of the continuity of the Oseledets filtration and decomposition as functions of the cocycle requires some preparation. 

We introduce (see the preamble to Section~\ref{continuity_oseledets}) a general topological space of measurable cocycles. We then allow perturbations of a given cocycle within the {\em whole space}.

We define spaces of measurable filtrations and decompositions and endow them with appropriate topologies (see Subsection~\ref{smf}). 

In the case of {\em higher dimensional} (i.e. $\Mat (m, \R)$-valued  with $m>2$) cocycles, as we perturb a given cocycle, the dimensions of the corresponding subspaces of its Oseledets filtration or decomposition {\em may change}. We define some natural projections / restrictions of these filtrations / decompositions, which will allow us to formulate and to prove stronger continuity results. In Subsection~\ref{direction} we establish the continuity of the most expanding direction, in Subsection~\ref{cof} that of the Oseledets filtration, and finally in Subsection~\ref{cod} we obtain the continuity of the Oseledets decomposition. 

We note that as with the Lyapunov exponents, our continuity results are {\em quantitative}. 

\medskip

To give an idea of these continuity results, we formulate here a simplified, particular version of our results in Section~\ref{continuity_oseledets}.

Let $(X, \mu, T)$ be an ergodic dynamical system with $T$ invertible. 

Let $\cocycles_m$ be a space of measurable cocycles $A \colon X \to \Mat (m, \R)$, endowed with a distance ($\dist$) at least as fine as the $L^\infty$-distance. 

We make the following assumptions:

\begin{enumerate}
\item[i.] The base dynamics satisfies an LDT estimate for a rich enough (relative to $\cocycles_m$) set of observables.
\item[ii.] Every cocycle $A \in \cocycles_m$ satisfies a uniform (relative to $\dist$) integrability condition.
\item[iii.] Every cocycle $A \in \cocycles_m$ with $L_1 (A) > L_2 (A)$ satisfies a fiber LDT which is uniform in a neighborhood of $A$.
\end{enumerate}

If $A \in \cocycles_m$ is such that $L_1 (A) > L_2 (A)$, then its Oseledets decomposition contains a one dimensional subspace $E_1 (A) (x)$ corresponding to the maximal Lyapunov exponent $L_1 (A)$. This defines (after identifying one dimensional subspaces with points in the projective space $\Proj$) a measurable function $E_1 (A) \colon X \to \Proj$.

By the continuity of the Lyapunov exponents established in Chapter 3 of~\cite{LEbook} (see also our preprint~\cite{LEbook-chap3}), if $A \in \cocycles_m$ is such that $L_1 (A) > L_2 (A)$, then for any nearby cocycle $B$ we have $L_1 (B) > L_2 (B)$. Hence $E_1 (B)$ is well defined as well, and we will prove the following.

\begin{theorem}\label{cont dec simple}
With the settings and assumptions described above, if $A \in \cocycles_m$ with $L_1 (A) > L_2 (A)$, then locally near $A$ the map
$$\cocycles_m \ni B \mapsto E_1 (B) \in L^1 (X, \Proj)$$
is continuous, with a modulus of continuity depending explicitly on the parameters of the LDT estimates.
In fact, a more precise pointwise statement holds. There are constants $\delta > 0$, $\alpha >0$ and a modulus of continuity function $\omega (h)$, all dependent only on $A$, such that for any cocycles $B_i$, $i=1, 2$ with $\dist (B_i, A) < \delta$,
$$\mu \ \{ x \in X \colon d (E_1 (B_1) (x), \, E_1 (B_2) (x) ) >  \dist (B_1, B_2)^\alpha \ \}  < \omega (\dist (B_1, B_2) )\,,$$
where as $h \to 0$, $ \omega (h) \to 0$  at a rate that depends explicitly on the LDT estimates.
\end{theorem}

This result (and the more general ones in Section~\ref{continuity_oseledets}) are applicable to both random (i.i.d. or Markov) irreducible cocycles  and to quasi-periodic cocycles, since LDT estimates will be established for these models (see Chapters 5 and 6 in~\cite{LEbook}).

Continuity of the Oseledets decomposition for $\GL (2, \C)$-valued random i.i.d. cocycles was obtained by C. Bocker-Neto and M. Viana in \cite{Bocker-Viana}. Their result is not quantitative but it requires no generic assumptions (such as  irreducibility) on the space of cocycles.
A different type of continuity property, namely stability of the Lyapunov exponents and of the Oseledets decomposition under {\em random perturbations} of a fixed cocycle, was studied in \cite{Led-Young, Ochs}.

\section{The ergodic theorems}
\label{bkoet}

We formulate the ergodic theorems of Birkhoff and Kingman,  then define the LE of a linear cocycle over a measurable bundle. We obtain a new proof of the multiplicative ergodic theorem of Oseledets using the AP.  

\subsection{The ergodic theorems of Birkhoff and Kingman}
\label{ket}
\newcommand{\Rbar}{ \overline{\mathbb{R}} }

The proofs of Birkhoff's pointwise ergodic theorem and Kingman's ergodic theorem can be found in most monographs covering topics in ergodic theory (see for instance \cite{Viana-book, Walters}). It is also worth mentioning in this context the simple proofs by Y. Katznelson and B. Weiss (see \cite{Katz-Weiss}). The method in \cite{Katz-Weiss} is based on a stopping time argument, an instance of which will appear in our proof of the MET in Subsection~\ref{met}, and it was also used in Chapter 3 of~\cite{LEbook} to establish a type of uniform upper semicontinuity of the maximal LE.
\begin{theorem}[Birkhoff's ergodic theorem]\label{bet}
Let $(X, \mu, T)$ be an ergodic dynamical system, and let $\xi \in L^1 (X, \mu)$ be an observable. Then 
\begin{equation*}\label{bet-eq}
\frac{1}{n}  \sum_{j=0}^{n-1} \xi (T^j x) \to \int_X \xi (x) \mu (d x) \quad \text{ for } \ \mu \text{ a.e. } x \in X.
\end{equation*}
\end{theorem}

\medskip

A sequence of numbers $\{a_n\}_{n\geq 0}$ in $[-\infty, +\infty)$ is called {\em sub-additive} if
$$ a_{n+m}\leq a_n + a_m \quad \text{ for all }\; n,m\geq 0  \;. $$ 

\begin{lemma}[Fekete's Subadditive Lemma]
\label{Kingmans lemma} Given a  sub-additive sequence $\{a_n\}_{n\geq 0}$ the following limit converges
$$ \lim_{n\to \infty} \frac{ a_n }{n}  
= \inf_{n\geq 1} \frac{a_n}{n}  
\in [-\infty,+\infty)\;. $$
\end{lemma}

\medskip


%

\begin{theorem}[Kingman's Ergodic Theorem]
Let $(X,\mu,T)$ be an ergodic dynamical system.
Given a sequence of measurable functions
$f_n:X\to\Rbar$ such that $f_1^+\in L^1(X,\mu)$ and
\begin{equation}\label{sub additivity}
 f_{n+m}\leq f_n + f_m\circ T^n \quad \text{ for all }\; n,m\geq 0 \;,
\end{equation}
then the sequence  $\{ \int f_n\,d\mu\}_{n\geq 0}$ is  sub-additive,  
and for $ \mu$-a.e. $x\in X$, $\frac{1}{n}\,f_n(x)$ converges to
the limit
$$ \lim_{n\to \infty} \frac{1}{n}\, \int f_n\,d\mu 
= \inf_{n\geq 1} \frac{1}{n}\, \int f_n\,d\mu 
\in [-\infty,+\infty)\;. $$
\end{theorem}

Let $\Bundle\subseteq X\times \R^m$ be a measurable bundle determined by some measurable function
$E:X\to \Gr(\R^m)$. This means that
$$ \mathcal{B}=\{\,(x,v)\,\colon\, x\in X,\, v\in E(x)\,\}\;.$$

We denote by $\mathcal{B} (x)$ the fiber over the base point $x$ and note that as a set, it coincides with the subspace $E (x)$.

\begin{definition}
\label{def linear cocycle}
A linear cocycle on $\mathcal{B}$ over a measure preserving dynamical system $(X,\mu,T)$ is a measurable map $F_A:\Bundle \to \Bundle$,  defined by a measurable family of linear maps
$A(x):E(x)\to E(T x)$, \, 
$F_A(x,v):=( T x, A(x) v )$.

We refer to $F_A$ as the cocycle $(T,A)$, or simply as $A$.
\end{definition}

\begin{definition}
\label{def integrable cocycle}
A cocycle $A$ is said to be {\em $\mu$-integrable}  if 
$$ \int_X \log^+ \norm{A(x)}\, d\mu(x)<+\infty \;.$$
\end{definition}

\begin{proposition} 
\label{prop  L1 = lim log An}
Given a $\mu$-integrable cocycle $A$,
 for $\mu$ almost every $x\in X$,
$$ L_1(A)= \lim_{n\to \infty} \frac{1}{n}\,
\log \norm{\An{n}(x)}  \;.$$
\end{proposition}

\begin{proof}
The sequence of functions $f_n:X\to \Rbar$,
$f_n(x)= \log \norm{\An{n}(x)}$ satisfies the sub-additivity property ~\eqref{sub additivity} and $f_1^+ = \log^+ \norm{A}\in L^1(X,\mu)$.

\end{proof}

\begin{proposition}\label{prop: LE and SVs}
If $A$ is a $\mu$-integrable cocycle then 
the following limit exists for  any $1\leq i\leq m$ and $\mu$-a.e. $x\in X$,
\begin{equation} 
\label{LE Li}
 \lim_{n\to \infty} \frac{1}{n}\,
\log s_i( \An{n}(x))  = L_i(A) \;.
\end{equation}
The number $L_i(A)\in [-\infty,+\infty)$ is called the
$i$-th Lyapunov exponent of  $A$.
Moreover, for all $2\leq i\leq m$,
\begin{equation}
\label{L1 Wedge i: 1}
 L_1(\wedge_i A)= L_i(A) + L_1(\wedge_{i-1} A)\;.
\end{equation}
\end{proposition}

\begin{proof}
Consider the exterior power cocycles $\wedge_i A$
where $1\leq i \leq m+1$.
Since
$$ \log \norm{\wedge_i \, A} \leq i\,\log \norm{A} \;,$$
the integrability  condition $\int_X \log^+ \norm{A}\, d\mu<+\infty$ for $A$
implies that all cocycles $\wedge_i \, A$ are also $\mu$-integrable.
Because $\wedge_{m+1} A(x)\equiv 0$ we have $L_1(\wedge_{m+1} \, A)=-\infty$.
Let $k$ be the first integer $1\leq j \leq m+1$ such that $L_1(\wedge_j \,A)=-\infty$. Then $L_1(\wedge_{k-1} \,A)>L_1(\wedge_{k} \,A)=-\infty$.
It is easy to see that for any matrix $g \in \Mat (m, \R)$ and for any $1\le i \le k$ we have  $\displaystyle s_i (g) = \frac{\wedge_i g}{\wedge_{i-1} g}$, hence 
$$ s_i(\An{n}(x)) = 
\frac{\norm{\wedge_i \,\An{n}(x)}}{\norm{\wedge_{i-1} \, \An{n}(x)}} \quad \text{ for all } 1 \le i \le k\;.$$
Note that $\norm{\wedge_{i-1} \,\An{n}(x)}$ is eventually non-zero because $L_1(\wedge_{i-1}\,A)>-\infty$. 
Hence, taking logarithms and applying Kingman's theorem,
 the limit~\eqref{LE Li} exists and the relation ~\eqref{L1 Wedge i: 1} holds.
Note also that for $i=k$ we get
$$ L_k(A)= L_1(\wedge_k A) - L_1(\wedge_{k-1}\, A)=-\infty \;.$$
For $k\leq i \leq m$, since $s_i(\An{n}(x))\leq s_k(\An{n}(x))$, by comparison
we infer that $L_i(A)=-\infty$ as well.

\end{proof}

\begin{corollary}\label{L1 wedge = sum Lj}
If $\int_X \log^+ \norm{A(x)}\, d\mu(x)<+\infty$ then 
 for $\mu$-a.e. $x\in X$, and  $1\leq i\leq m$,
 $$   \lim_{n\to \infty} \frac{1}{n}\,
\log \norm{\wedge_i \, \An{n}(x)}  = L_1(A) + \ldots + L_i(A) \;. $$
\end{corollary}

\begin{proof}
Apply proposition~\ref{prop: LE and SVs},
using~\eqref{L1 Wedge i: 1} inductively. 
\end{proof}

\subsection{Review of Grassmann geometry concepts and notations}
\label{notations}


What follows is an outline of the notions described in Chapter 2 of~\cite{LEbook} (see also~\cite{LEbook-chap2}) which are needed in this paper. 

A sequence of integers $\tau=(\tau_1,\ldots, \tau_k)$ with
$1\leq \tau_1 <\tau_2<\ldots <\tau_k<m$  is called a signature.
Let $$s_1(g)\geq s_2(g)\geq \ldots \geq s_m(g)\geq 0$$
denote the ordered (repeated) singular values of a 
matrix $g\in\Mat(m,\R)$. We say that $g$ has a singular spectrum with a $\tau$-gap pattern, or shortly that it has a $\tau$-gap pattern, when
$$s_{\tau_j}(g)>s_{\tau_j+1}(g) \quad  \text{ for all } j=1,\ldots, k\,.$$ We say that it has an exact $\tau$-gap pattern when furthermore  
$s_{\tau_{j}+1}(g)=s_{\tau_{j+1}}(g)$ for all $j=0,1,\ldots, k$, with the conventions $\tau_0=0$ and $\tau_k=m$.

\medskip

Analogously, let $$L_1(A)\geq L_2(A)\geq \ldots \geq L_m(A)\geq -\infty$$ denote the ordered (repeated) Lyapunov exponents of a linear cocycle $A$.
We say that  $A$ has a Lyapunov spectrum with a $\tau$-gap pattern, or shortly that it has a $\tau$-gap pattern, when
$$L_{\tau_j}(A)>L_{\tau_j+1}(A) \quad \text{ for all } j=1,\ldots, k\,.$$ 
We say that it has an exact $\tau$-gap pattern when furthermore  
$L_{\tau_{j}+1}(A)=L_{\tau_{j+1}}(A)$ for all $j=0,1,\ldots, k$, with the conventions $\tau_0=0$ and $\tau_k=m$.

\medskip

Given a matrix $g\in\Mat(m,\R)$ with singular value gap ratio 
$$\rgap (g) := \frac{s_1(g)}{s_2(g)} > 1\,,$$
its   most expanding direction is the point
 $\mostexp(g)\in \Proj$  determined by any singular vector of $g$ associated to the first singular value $s_1(g)=\norm{g}$.

 More generally, if $1\leq k\leq m$ is such that $$\rgap_k (g) := \frac{s_k(g)}{s_{k+1}(g)} > 1\,,$$ the most expanding $k$-plane is the 
 $k$-dimensional vector subspace  $\mostexp_k(g)$  spanned by the singular vectors of $g$ associated to the first $k$ singular values of $g$. 
 
 Finally, when $g$ has a $\tau$-gap pattern, hence  $$\rgap_\tau (g) := \min_{1\le j \le k} \rgap_{\tau_j} (g) >1\,,$$  we  define the $\tau$-flag
 $$\mostexp_\tau(g)=(\mostexp_{\tau_1}(g),\ldots,  \mostexp_{\tau_1}(g))\in\FF_\tau(\R^m)\,.$$
 
 Given matrices $g_0, g_1 \in \Mat (m, \R)$, we define their {\em expansion rift} is
  $$\rift (g_0, g_1) := \frac{\norm{g_1 \, g_0}}{\norm{g_1} \norm{g_0}}\,.$$

\medskip

A $\tau$-flag in $\R^m$ is a finite strictly increasing sequence $F=(F_1,\ldots, F_k)$ of vector subspaces
$F_1\subset F_2\subset \ldots \subset F_k\subset \R^m$ such that $\dim F_j=\tau_j$ for all $j=1,\ldots, k$. The space of all $\tau$-flags in $\R^m$ is denoted here by $\FF_\tau(\R^m)$.

The orthogonal complement $F^\perp$ of a flag $F=(F_1,\ldots, F_k)$ is the flag $F^\perp=(F_k^\perp,\ldots, F_1^\perp)$ of its orthogonal complements, which has the complementary signature $\tau^\perp=(m-\tau_k,\ldots, m-\tau_1)$.

A $\tau$-decomposition of $\R^m$ is a family $E_{\cdot}=\{E_j\}_{1\leq j\leq k+1}$ of vector subspaces such that $\R^m=\oplus_{j=1}^{k+1} E_j$, and $\dim E_j=\tau_j-\tau_{j-1}$  for all $j=1,\ldots, k+1$,
again with the conventions $\tau_0=0$ and $\tau_k=m$.
We denote by $\Decompsp_\tau(\R^m)$ the space of all $\tau$-decompositions of the Euclidean space $\R^m$.

Given two flags $F\in\FF_\tau(\R^m)$ and $F\in\FF_{\tau^\perp}(\R^m)$, of complementary signatures, the quantity $\theta_{\sqcap}(F,F')$ measures the transversality between
each subspace $F_j$ in $F$ and the corresponding subspace
$F_{k-j+1}$ in $F'$. When $\theta_{\sqcap}(F,F')>0$
all these pairs $(F_j,F_{k-j+1})$ of subspaces have a transversal intersection and the following family of subspaces
$ F\sqcap F' =\{ F_j\cap F_{k-j+2}'\}_{1\leq j\leq k+1}$
is a $\tau$-decomposition (see Proposition 2.31 in Chapter 2 of~\cite{LEbook} or Proposition 3.14 in \cite{LEbook-chap2}).


The following table of notations provides references to these concepts as defined in our preprint~\cite{LEbook-chap2}.

\begin{table}[h!]
\begin{center}
\begin{tabular}{ccc}
Concept \qquad\quad & Takes values in \; & Definition \\
\hline
$\mostexp(g)$ & $\Proj$ &  2.6\\
$\mostexp_k(g) $ & $\Gr_k(\R^m)$ &  2.7\\
$\mostexp_\tau(g)$ & $\FF_\tau(\R^m)$ &  2.10\\
$F^\perp$ & $\FF_{\tau^\perp}(\R^m)$ &  1.7\\
$F\sqcap F'$ & 
$\Decompsp_\tau(\R^m)$ &  3.5 \\
$\theta_\sqcap(F, F')$ & 
$\R$ &   3.4
\end{tabular}
\end{center}
\end{table}

Finally, let us formulate the statement in the AP that will be used in this paper (see Proposition 2.37 in~\cite{LEbook} or Proposition 4.2 in~\cite{LEbook-chap2}).

\begin{proposition} \label{AP-practical}
There exists $c>0$ such that
 given $0<\epsilon<1$,  $0<\kappa\leq c\,\epsilon^ 2$ 
and  \,  $g_0, g_1,\ldots, g_{n-1}\in\gl(m,\R)$, \,
 if
\begin{align*}
\rm{(gaps)} \  & \rgap (g_i) >  \frac{1}{\ka} &  \text{for all }  & \ \  0 \le i \le n-1  
\\
\rm{(angles)} \  & \frac{\norm{ g_i \cdot g_{i-1} }}{\norm{g_i}  \, \norm{ g_{i-1}}}  >  \ep  & \
 \text{for all }   & \  \ 1 \le i \le n-1  
\end{align*}
then  
\begin{align*} 
& \max\left\{ \, d(\mostexp(g^{(n)\ast}), \mostexp(g_{n-1}^\ast)),\,
d(\mostexp(g^{(n)}), \mostexp(g_{0})) \, \right\}
 \lesssim    \kappa\,\ep^{-1}\,. \\
\end{align*}
 \end{proposition}

\subsection{The multiplicative ergodic theorem}
\label{met}

\newcommand{\kaAP}{ \kappa_{ap} }
\newcommand{\epAP}{ \epsilon_{ap} }
\newcommand{\gap}{{\rm gap}}

Throughout this section let $T:X\to X$ be an ergodic invertible measure preserving transformation on a  probability space  $(X,\mathcal{F},\mu)$. 

Consider a measurable bundle $\Bundle\subseteq X\times \R^m$  determined by some measurable function $E:X\to \Gr(\R^m)$, and  $\mu$-integrable linear cocycle $F_A:\Bundle \to \Bundle$,
defined by a measurable family of linear maps
$A(x):E(x)\to E(T x)$.

Given a vector $v\in E(x)$ 
we define
\begin{align*}
 \lambda_A(x,v):= \limsup_{n\to+\infty}\frac{1}{n}\,\log \norm{\An{n}(x) v} \;,\\
  \lambda_A^-(x,v):= \liminf_{n\to+\infty}\frac{1}{n}\,\log \norm{\An{n}(x) v} \; .
\end{align*}

Notice that $\lambda_A(x,0)=-\infty$.
This function satisfies the following properties:

\begin{proposition}\label{lambda LE properties}
For every $x\in X$, given vectors $v,v'\in E(x)$,
\begin{enumerate}
\item[(a)] \; $\displaystyle \lambda_A(x,v)\leq L_1(A)$,
\item[(b)] \; $\displaystyle \lambda_A(x,c \,v)= \lambda_A(x,v)$\, if $c\neq 0$,
\item[(c)] \; $\displaystyle\lambda_A(x,v+v') \leq  \max\{\lambda_A(x,v),\lambda_A(x,v') \}$,
\item[(d)]\;  
if \, $\lambda_A(x,v')< \lambda_A^-(x,v)=\lambda_A(x,v)$ then 
$$\lambda_A^-(x,v + v') = \lambda_A(x,v + v') = \lambda_A(x,v)\,.$$
\item[(e)] \; $\displaystyle \lambda_A(x,v) = \lambda_A(T x, A(x)\,v)$.
\end{enumerate}
\end{proposition}

\begin{proof}
Item (a) follows from the inequality $\norm{\An{n}(x)\,v}\leq \norm{\An{n}(x)}\,\norm{v}$.
Item (b) is a straightforward consequence of the definition.
Item (c) follows from the inequality 
\begin{align*}
\log \norm{\An{n}(x)(v+v')} &\leq 
\log \left( \norm{\An{n}(x) v} + 
 \norm{\An{n}(x)v' } \right) \\
&\leq \log \left( 2\,\max\{ \norm{\An{n}(x) v}, \norm{\An{n}(x)v' }\} \right)\\
& = \log 2 + \max\{ \log \norm{\An{n}(x) v},
\log \norm{\An{n}(x)v' }\}\;.
\end{align*}   
Item (d) follows from the inequality 
\begin{align*}
\norm{\An{n}(x) v} \left( 1-\frac{\norm{\An{n}(x) v'}}{\norm{\An{n}(x) v}}\right) \leq 
\norm{\An{n}(x) (v+v') } \\
\qquad \leq 
\norm{\An{n}(x) v} \left( 1+\frac{\norm{\An{n}(x) v'}}{\norm{\An{n}(x) v}}\right)
\end{align*}
and the fact that 
$$\limsup_{n\to+\infty}\frac{1}{n}	\log\norm{\An{n}(x) v'} < \lim_{n\to+\infty}\frac{1}{n}	\,\log\norm{\An{n}(x) v}$$ 
implies the ratio $\norm{\An{n}(x) v'}/\norm{\An{n}(x) v}$ converges geometrically to $0$.
Finally, item (e) follows from the identity
$ \An{n}(x) v  =  \An{n-1}(T x) (A(x) v)$.

\end{proof}

Given a real number $\lambda\in\R$, the set
$$ F_\lambda(x):=\{\, v\in E(x)\,\colon\,
\lambda_A(x,v)\leq \lambda\,\} \;, $$
is a linear subspace of $E(x)$, because of items (b) and (c) of the previous proposition.
This family of subspaces determines a finite filtration (flag)
$$ \{0\}\subsetneq F_{\lambda_1}(x)  
\subsetneq F_{\lambda_2}(x)    \ldots \subsetneq F_{\lambda_k}(x) \subsetneq  F_{\lambda_{k+1}}(x)=E(x)    $$
which by item (e) is invariant in the sense that
$A(x)\,F_\lambda(x)\subseteq F_\lambda(T x)$, for all $x\in X$.
The multiplicative ergodic theorem (MET) gives a precise description of this filtration and its relation with the Lyapunov exponents.

Assume that $A$ is $\mu$-integrable
and  $ L_1(A)>L_2(A)$.
The following proposition shows the existence 
of a measurable function $\mostexp^{\infty}(A):X\to\Proj$ with the {\em most expanding direction} of the cocycle $A$.

For each $n\in\N$ we define a partial function
$$ \mostexp^{(n)}(A)(x):=
\left\{ \begin{array}{lll}
\mostexp(A^{(n)}(x)) & & \text{ if }  \rgap(A^{(n)}(x))>1 \\
\text{undefined} &  &  \text{otherwise}. 
\end{array} \right. $$

\medskip

\begin{definition}\label{Cauchy}
Let $(Y,d)$ be a metric space.
We say that a sequence of partial functions $f_n:D_n\subseteq X\to Y$ is {\em $\mu$ almost everywhere Cauchy} if given $\epsilon>0$ there exists a set $B\in\mathscr{A}$ with $\mu(B)<\epsilon$ and $n_0\in\N$ such that for all $n\geq n_0$,
the function $f_n(x)$ is well-defined on $X\setminus B$, i.e., $ X\setminus B\subseteq D_n$, and the sequence
$\{ f_n(x)\}_{n\geq n_0}$ is Cauchy for every $x\notin B$. 
\end{definition}

\medskip

\begin{proposition}\label{med:aec}
Let $A$ be a $\mu$-integrable cocycle such that
 $L_1(A)>L_2(A)$. The sequence of (partial) functions $ \mostexp^{(n)}(A)$ from $X$ to $\Proj$  is $\mu$ almost everywhere Cauchy. In particular, it  converges $\mu$ almost everywhere to a (total) measurable function $\mostexp^{(\infty)}(A):X\to\Proj$.
Moreover,  for $\mu$-a.e. $x\in X$,
$$\limsup_{n\to +\infty}\frac{1}{n} \,\log \, d(\mostexp^{(n)}(A)(x), \mostexp^{(\infty)}(A)(x)) \leq L_2(A)-L_1(A) <0 \;.$$
\end{proposition}

This proposition will be proved using the Avalanche Principle.

\begin{lemma}\label{doubling seq}
Given $\epsilon>0$ there exists  $r\in\N$ such that for any $n,n_0\in\N$ with $n\geq  r\,n_0$
there is a sequence of integers $\{m_i\}_{i\geq 0}$ for which
\begin{enumerate}
\item[(a)] $m_0=n_0$, 
\item[(b)] $m_k=n$ \, for some $k\geq 1$, and
\item[(c)] 
$\abs{m_i- 2\,m_{i-1}}<\epsilon\,m_{i}$ for all $i\geq 1$.
\end{enumerate}

\end{lemma}

\begin{proof}
Choose $k\geq 1$ such that $2^k \leq n/n_0 < 2^{k+1}$,
and define $\theta = \frac{1}{k}\,\log_2(n/n_0)-1$, so that
$0\leq \theta <\frac{1}{k}$. The sequence $m_i:=\lfloor n_0\,2^{(1+\theta)i}\rfloor$ satisfies (a) and (b).
From
$$ n_0\,2^{(1+\theta)i}-1 < m_i \leq n_0\,2^{(1+\theta)i}
\quad \text{ for all }\; i\geq 0\;,$$
we obtain 
$$   \frac{\abs{m_i-2\, m_{i-1}}}{m_i} <\frac{\log 2}{k} +\frac{2}{m_i}\;.$$ 
Item (c) follows choosing $r=2^l$ where $l\in\N$ is such that
$\frac{\log 2}{l}+\frac{1}{2^{l-1}}<\epsilon$.

\end{proof}

\begin{definition}
Given $\epsilon>0$, we call an $\epsilon$-doubling sequence any sequence $\{m_i\}_{i\geq 0}$
 of integers such that 
$\abs{m_i- 2\,m_{i-1}}<\epsilon\,m_{i}$ for all $i\geq 1$.
\end{definition}

\begin{lemma} \label{lemma recurrent doubling sequence}
Given $\epsilon>0$ small enough and a measurable set $\Omega\subset X$ such that $\mu(\Omega)> 1-\epsilon/4$ there is a measurable subset $\Omega_0\subseteq \Omega$ 
with $\mu(\Omega_0)> 1-\epsilon$, and $r, n_0\in\N$ such that
\begin{enumerate}
\item[(a)] \; For each $x\in\Omega_0$  there is a $\epsilon$-doubling sequence $\{m_i\}_{i\geq 0}$ satisfying  $m_0=n_0$ and  $T^{m_i} x\in\Omega$ for all $i\geq 0$;
\item[(b)] \; For all $x\in\Omega_0$  and $n\geq r\, n_0$, there is an $\epsilon$-doubling sequence $\{m_i\}_{i\geq 0}$ satisfying  $m_0=n_0$, $m_k=n$ for some $k\geq 1$,  and  $T^{m_i} x\in\Omega$ for all $0\leq i <k$.
\end{enumerate}
\end{lemma}

\begin{proof}
By Birkhoff's ergodic theorem, for
 $\mu$-a.e. $x\in X$,
\begin{equation}\label{limit return frequency}
\lim_{m\to \infty}\, \frac{1}{m}\,
\#\{\, 0\leq j\leq m-1\,\colon\,
T^j (x)\notin \Omega 
\,\} = \mu(X\setminus \Omega)  <\frac{\epsilon}{4} \;.  
\end{equation} 

Given a phase $x$, if we denote by $m (x)$ the first integer such that the inequality
$$\frac{1}{m}\,
\#\{\, 0\leq j\leq m-1\,\colon\,
T^j (x)\notin \Omega 
\,\}  < \frac{\epsilon}{4} $$
holds for all $m \ge m (x)$, then by \eqref{limit return frequency}, $m (x)$ is defined for $\mu$-a.e. $x  \in X$.

For every integer $n$, let $\U_n := \{ x \in \Omega \colon m (x) \le n \}$. Since $\U_n \subset \U_{n+1}$ and $\cup_n \, \U_n$ has full (relative) measure in $\Omega$, there is $n_0 = n_0 (\epsilon)$ such that $\mu (\Omega \setminus \U_{n_0}) < \epsilon/2$. 

Note that if $x \in \U_{n_0}$, then
\begin{equation}\label{return frequency}
\# \{\, 0\leq j\leq m-1\,\colon\,
T^j (x)\notin \Omega 
\,\} < \frac{\epsilon \,m}{4}\quad \text{ for all } \; m\geq n_0 \;.
\end{equation}

We set $\Omega_0 := \U_{n_0} \cap T^{-n_0}(\Omega)$. Then 
$$\mu(X\setminus \Omega_0)\leq \mu(X\setminus \Omega) + \mu(\Omega \setminus \U_{n_0}) +  \mu(X\setminus T^{-n_0}(\Omega)) < \epsilon \;,$$ 
and if $x\in\Omega_0$ then \eqref{return frequency} holds and  $T^{n_0} x\in \Omega$.

\smallskip

To prove (a), take $x\in\Omega_0$ and consider the sequence $a_i :=2^i\, n_0$.

For each $i\geq 1$, applying~\eqref{return frequency} with $m=a_i$,   there is an integer $m_i$ in the range  $(1-\epsilon/4) a_i\leq m_i \leq a_i$ such that $T^{m_i} x\in \Omega$.
A straightforward computation shows that $\{m_i\}_{i\geq 0}$ is an $\epsilon'$-doubling sequence with $\epsilon'=\frac{\epsilon/4}{1-\epsilon/4}<\epsilon$.

\smallskip

Finally, to prove (b), we fix $x \in \Omega_0$ and use Lemma~\ref{doubling seq} to get an integer $r=r(\epsilon)$ so that 
 if $n\geq r \, n_0$, there is an $\frac{\epsilon}{4}$-doubling sequence $\{m_i\}_{i\geq 0}$ with  $m_0 = n_0$ and $m_k = n$ for some index $k \ge 1$.
By the frequency bound~\eqref{return frequency} applied with $m = m_i + \frac{\epsilon \, m_i}{6}$, for each $1\leq i \leq k-1$ there is $m_i'\in\N$ such that  $T^{m_i'} x\in \Omega$ and $\abs{m_i-m_i'}<\epsilon\, m_i/6$.
Setting $m_0'=m_0=n_0$ and $m_k'=m_k=n$,
 the sequence $\{m_i'\}_{i\geq 0}$ satisfies (b), and a simple calculations shows that it is $\epsilon$-doubling. 

\end{proof}

The next proposition says that for any given $\epsilon>0$ there is a measurable set of phases $\Omega_0$ with  
$\mu(\Omega_0)>1-\epsilon$,   such that if $x\in\Omega_0$
then there exists an $\epsilon$-doubling sequence of  {\em avalanche times}, that is, times where the assumptions of the AP hold.

\begin{proposition} \label{prop recurrent doubling sequence and AP estimates}
Let $A$ be a $\mu$-integrable cocycle with
 $L_1(A)>L_2(A)$. 
 
 Given $0<\varkappa <L_1(A)-L_2(A)$, and $0<\epsilon\ll \varkappa$, there exist $r, n_0\in\N$ and a measurable set $\Omega_0\subset X$ with  $\mu(\Omega_0)> 1-\epsilon$ and  such that for any $x\in\Omega_0$ and $n\geq r\,n_0$, there is a $\epsilon$-doubling sequence $\{m_i\}_{i\geq 0}$  with $m_0=n_0$, $m_k=n$ for some $k\geq 1$, and  such that for all $i\geq 0$,
	\begin{enumerate}
	\item[(1)] 
	$\rgap( \An{m_{i}}(x) )  \geq e^{ m_{i}(\varkappa-2\epsilon)}$ \ and \\ \\
	$\rgap( \An{m_{i+1} - m_{i}}(T^{m_{i}}x) )  \geq e^{  m_{i} (\varkappa-2\epsilon) (1-\ep) / (1+\ep)}$.\\
	
	\item[(2)] 
	$\rift (\An{m_{i}}(x),\,
	\An{m_{i+1} - m_{i}}(T^{m_{i}}x) )  \geq e^{- 5 m_{i}  \epsilon }$.
	\end{enumerate}
\end{proposition}

\begin{proof}
The following limits exist for $\mu$-a.e. $x\in X$,
\begin{align*}
&  \lim_{n\to\infty}\frac{1}{n}\,
\log \norm{A^{(n)}(x)} = L_1(A) \;,\\
&  \lim_{n \to\infty}\frac{1}{n}\,
\log  \norm{\wedge_2 \, A^{(n)}(x)}  = L_1(A)+ L_2(A) < 2 L_1 (A) - \varkappa \;.
\end{align*}

Take $0<\epsilon\ll \varkappa$ small. For any $n_0\in\N$ consider the measurable set $\Omega_{n_0}(\epsilon)$ of $x\in X$ such that for all $n\geq \frac{n_0}{2}$ we have
\begin{equation}\label{convergence:conditions}
  e^{  n \,(L_1(A) - \epsilon) }\leq \norm{A^{(n)}(x)} \leq e^{  n \,(L_1(A)+\epsilon) }  \ \text{ and } \  \norm{\wedge_2\, A^{(n)}(x)} \leq e^{n\,( 2 L_1(A)-\varkappa )}  \;.
\end{equation}
The almost sure convergence of the above functions implies that
$$ \lim_{n\to+\infty} \mu(\Omega_{n}(\epsilon))= 1\;.$$
We assume that $n_0$ is  large enough that also $\mu(X\setminus  \Omega_{n_0}(\epsilon))<\epsilon/2$.

Setting $\Omega :=\Omega_{n_0}(\epsilon)$, by Lemma~\ref{lemma recurrent doubling sequence}
there are integers $r$ and $n_0' > n_0$,
 and a measurable subset 
$\Omega_0\subset \Omega$ such that  for all $x\in\Omega_0$  and $n\geq r\, n_0'$, there is an $\epsilon$-doubling sequence $\{m_i\}_{0\leq i \leq k}$ satisfying  $m_0=n_0$, $m_k=n$,  and  $T^{m_i} x\in\Omega_{n_0}(\epsilon)$ for all $0\leq i <k$.

Item (1) follows from the fact that if $x\in \Omega_{n_0}(\epsilon)$,  then
\begin{equation*}
\rgap( \An{n}(x) ) = \frac{\norm{\An{n}(x)}^2}{\norm{\wedge_2\, \An{n}(x)}} \geq e^{n\,(\varkappa-2\epsilon)}\quad \text{ for all } \; n\geq \frac{n_0}{2} \;.
\end{equation*}

Applying the estimate above with $n :=m_i$ yields the first inequality in item (1), while the second follows by putting $ n := m_{i+1} - m_i$. Note that the $\ep$-doubling condition implies that 
$m_{i+1} - m_i > \frac{1-\ep}{1+\ep} \, m_i \ge \frac{1}{2} \, n_0$.

For item (2) we use~\eqref{convergence:conditions}. Since  $x, T^{m_i} x\in\Omega_{n_0}(\epsilon)$,  
\begin{equation*} 
\frac{\norm{A^{(m_{i+1})}(x)}}{\norm{A^{(m_i)}(x)}\,\norm{A^{(m_{i+1}-m_i)}(T^{m_i} x)}}\geq e^{-2\,\epsilon\,m_{i+1}}  \geq  e^{-5\,\epsilon\,m_{i}} \,.
\end{equation*}

This completes the proof, as the left hand side of the inequality represents the rift $\rift (\An{m_{i}}(x),\,
	\An{m_{i+1} - m_{i}}(T^{m_{i}}x) )$.

\end{proof}

\begin{proof}\emph{(of Proposition~\ref{med:aec})}
We use Proposition~\ref{prop recurrent doubling sequence and AP estimates}. Fix $x \in \Omega_0$, and 
for each $i\geq 0$ we apply the avalanche principle to the sequence of two matrices $g_0=\An{m_i}(x)$ and $g_1=\An{m_{i+1} - m_i}(T^{m_i} x)$, noting that $g_1 \, g_0 = \An{m_{i+1}} (x)$.

The key parameters in this application of the AP are  
$$\kaAP = e^{ -  m_{i} (\varkappa-2\epsilon) (1-\ep) / (1+\ep)} < e^{  - m_{i} (\varkappa - 2 \epsilon) / 2} \quad \text{ and } \quad
\epAP = e^{-5\,\epsilon\, m_i }.$$

Note that $ \frac{\kaAP}{\epAP^2} <  e^{-  m_i (\varkappa / 2 -6 \, \epsilon)} \ll 1$. 

Moreover, item (1) and (2) in  Proposition~\ref{prop recurrent doubling sequence and AP estimates} imply the gap and angle conditions of the AP. Therefore, the avalanche principle (Proposition~\ref{AP-practical}) is applicable and we get:
$$d( \mostexp(\An{m_i}(x)), \,  \mostexp(\An{m_{i+1}}(x)) ) <  \frac{\kaAP}{\epAP} = e^{- m_i \, \theta},$$
where $\theta := (\varkappa - 2 \epsilon)  \, \frac{1-\ep}{1+\ep} - 5 \ep$. Note that as $\ep \to 0$ we have $\theta \to \varkappa$.

By the definition of an $\epsilon$-doubling sequence we have $m_{i+1}\geq \frac{2}{1+\epsilon} \, m_i > \frac{3}{2} \, m_i$, 
hence for all $i\ge0$ we have $m_{i}\geq (3/2)^i \, n_0$.  We then conclude:

\begin{align*}
d( \mostexp(\An{n_0}(x)), \,  \mostexp(\An{n}(x)) ) &=
d( \mostexp(\An{m_0}(x)), \,  \mostexp(\An{m_k}(x)) ) \\
&\leq \sum_{i=0}^{k-1}
d( \mostexp(\An{m_i}(x)), \,  \mostexp(\An{m_{i+1}}(x)) )\\
&\leq \sum_{i=0}^{k-1}
e^{ - m_i\, \theta} 
\leq \sum_{i=0}^{k-1}
e^{ - (3/2)^i\,n_0\, \theta} \less e^{ - n_0\, \theta}.
\end{align*}

 Taking $n_0$ large enough, the bound $e^{ - n_0 \, \theta}$ becomes arbitrarily small.
 This proves that the sequence $\{ \mostexp(\An{n}(x))\}_{n\geq n_0}$ is Cauchy.
 Moreover, passing to the limit as $n\to+\infty$,
$$ d(\mostexp(A^{(n_0)}(x)), \medir{\infty}(A)(x) ) \lesssim 
e^{- n_0 \, \theta}\;. $$

Therefore, as  $n=n_0$ is arbitrary,
$$ \limsup_{n \to+\infty} \frac{1}{n}\,\log d(\medir{n}(A)(x)), \medir{\infty}(A)(x) ) \le - \theta\,.$$

Finally, since  $\epsilon>0$ can be taken arbitrarily small,
and $\varkappa$ can be taken arbitrarily close to $L_1(A)-L_2(A)$, we conclude that

\qquad $\displaystyle \limsup_{n\to +\infty}\frac{1}{n} \,\log \, d(\mostexp^{(n)}(A)(x), \mostexp^{(\infty)}(A)(x)) \leq L_2(A)-L_1(A)  $.

\end{proof}

\medskip

Given $0\leq k\leq m$,  we define a sequence of partial functions
$\mostexp^{(n)}_k(A)$ on $X$ taking values in $\Gr_k(\R^m)$,
$$ \mostexp^{(n)}_k(A)(x):=
\left\{ \begin{array}{lll}
\mostexp_k(A^{(n)}(x)) & & \text{ if }  \rgap_k(A^{(n)}(x))>1\,, \\
\text{undefined} & & \text{otherwise}.
\end{array} \right. $$

\medskip

\begin{proposition}\label{prop: medir k exists}
If $L_k(A)>L_{k+1}(A)$ then the sequence of partial functions 
$\mostexp^{(n)}_k(A)$ from $X$ to $\Gr_k(\R^m)$ is almost everywhere Cauchy.  In particular, it  converges $\mu$ almost everywhere to a (total) measurable function $\medir{\infty}_k(A):X\to\Gr_k(\R^m)$.
Moreover,  for $\mu$-a.e. $x\in X$,
$$\limsup_{n\to +\infty}\frac{1}{n} \,\log \, d(\medir{n}_k(A)(x), \medir{\infty}_k(A)(x)) \leq L_{k+1}(A)-L_k(A) <0 \;.$$
\end{proposition}

\begin{proof}
Apply proposition~\ref{med:aec}
to the cocycle $\wedge_k A$. 
\end{proof}

%
%

\medskip

\begin{definition} \label{def: regular point}
Given a $\mu$-integrable cocycle $A$, we say that $x\in X$ is a $\mu$-regular point if whenever $L_j(A)> L_{j+1}(A)$, we have  
\begin{align*}
& \lim_{n\to +\infty}\frac{1}{n}\,\log \norm{\wedge_j\, \An{n}(x)} = 
L_1(\wedge_j \, A) \quad \text{ for }\; 1\leq j\leq m  \\
&  \limsup_{n\to +\infty}  \frac{1}{n}\,\log d\left( \medir{n}_j(A)(x),\,  \medir{\infty}_j(A)(x)\right) \le L_{j+1}(A)-L_j(A)\,. 
\end{align*}
\end{definition}

\medskip

\begin{proposition}
The set of $\mu$-regular points of a cocycle 
has full $\mu$-measure.
\end{proposition}

\begin{proof}
Combine corollary~\ref{L1 wedge = sum Lj} and proposition~\ref{prop: medir k exists}.

\end{proof}

\begin{definition}\label{singular basis}
Given a linear map $g:V\to V'$ between Euclidean spaces $V$ and $V'$ of dimension $m$, we call {\em singular basis of $g$} any orthonormal basis $\{ v_j\}_{1\leq j\leq m}$ of $V$ consisting of singular vectors $v_j$ of $g$ such that $\norm{g\, v_j}= s_j(g)$ for all $j=1,\ldots, m$.
\end{definition}
Note that for every $1\leq  k\leq m$, the unit $k$-vector $v_1\wedge \cdots \wedge v_k$ is a most expanding vector of $\wedge_k\, g$.

%

\medskip

\begin{proposition} 
\label{prop:lim mostexp infty = Lk}
Consider a $\mu$-integrable cocycle $A$, and a $\mu$-regular point $x\in X$.
If $\gamma= L_1(A)=  \cdots = L_k(A)>L_{k+1}(A)$  
then for any $v\in \medir{\infty}_k(A)(x)\setminus\{0\}$, we have
 $$ \lim_{n\to+\infty}\frac{1}{n}\,\log\norm{\An{n}(x)\,v}
 = \gamma \;.$$
 
 In particular, \ $ \lambda_A(x,v) = \lambda_A^-(x,v) = \gamma$.
\end{proposition}

\begin{proof}
Consider a singular basis $\{v_{1,n},\ldots, v_{m,n}\}$ for the linear map $\An{n}(x)$.
Let  $\{v_1,\ldots, v_k\} \subset \medir{\infty}_k(A)(x)$ be an orthonormal family  obtained as  limit of the sequence $\{v_{1,n_s},\ldots, v_{k,n_s}\}$, for some subsequence of integers $n_s$.  

Let $w_n= v_{1,n}\wedge \ldots \wedge v_{k,n}$
and $w= v_1 \wedge \ldots \wedge v_k$. 

After possibly changing the sign of $v_{1,n}$, and since  (by Proposition~\ref{med:aec} and the fact that $x$ is $\mu$-regular)  in the projective space $\hat{w_n} \to \hat{w}$, we have 
$w_n\to w$ as $n\to+\infty$.  Then
\begin{align*}  
\abs{ \frac{ \norm{\wedge_k \, \An{n}(x)\,w} }{ \norm{ \wedge_k \, \An{n}(x)\,w_{n}} } - 1}
&= \frac{ \abs{ \norm{\wedge_k \, \An{{n}}(x)\,w}  - \norm{\wedge_k \, \An{{n}}(x)\, w_n} }}{ \norm{\wedge_k \, \An{n}(x)} } \\
&  \leq  \frac{ \norm{\wedge_k \,\An{{n}}(x)\,w   -   \wedge_k \,\An{{n}}(x)\,w_{n}} }{ \norm{\wedge_k \,\An{n}(x)} } \\
& \leq \norm{ w   -  w_{n}}  \to  0 \quad \text{ as } n \to \infty \,.
\end{align*}

Since
$$  \norm{\wedge_k \, \An{n}(x)\,w}\leq \prod_{j=1}^k 
\norm{\An{n}(x)\,v_j} \;, $$
we have
\begin{align*}
k\,\gamma &= L_1(A)+\cdots + L_k(A) =
 \lim_{n\to \infty} \frac{1}{n}\,
\log \norm{\wedge_k \,\An{n}(x)} \\
& =
\lim_{n\to \infty} \frac{1}{n}\,
\log \norm{\wedge_k \, \An{n}(x)\,w_{n}} =
\lim_{n\to \infty} \frac{1}{n}\,
\log \norm{\wedge_k \, \An{n}(x)\,w}  \\
&\leq  \liminf_{n\to \infty} \frac{1}{n}\,
\sum_{j=1}^k \log \norm{\An{n}(x)\,v_j}
\leq \limsup_{n\to \infty} \frac{1}{n}\,
 \sum_{j=1}^k \log \norm{\An{n}(x)\,v_j} \\
&\leq  \lim_{n\to \infty} \frac{1}{n}\,
\sum_{j=1}^k \log \norm{\An{n}(x)} = k\, \gamma  \;. 
\end{align*}

Thus, the sequences
\ $ \displaystyle c_{i,n}:= \frac{1}{n}\,\log \norm{\An{n}(x)} - 
\frac{1}{n}\,\log \norm{\An{n}(x)\,v_i}\geq 0
$ \ 
satisfy
$$ 0\leq c_{i,n} \leq \sum_{j=1}^k c_{j,n} =
\frac{k}{n}\,\log\norm{\An{n}(x)}-\frac{1}{n}\,\sum_{j=1}^k \log \norm{\An{n}(x)\,v_j}\;. $$
But since the right-hand-side converges to $0$,
we obtain that for all $j=1,\ldots, k$,
$\lim_{n\to +\infty} c_{j,n}=0$,
or, equivalently, that
 $$ \lambda_A(x,v_j)= \lim_{n\to+\infty}\frac{1}{n}\,\log\norm{\An{n}(x)\,v_j}= \gamma \;.$$
 
 Now, given $v\in \medir{\infty}_k(A)(x)\setminus\{0\}$,
assume, by contradiction, that there exists a sequence $n_s\to +\infty$ such that 
\begin{equation}\label{absurd assumption}
\lim_{n\to+\infty}\frac{1}{n_s}\,\log\norm{\An{n_s}(x)\,v} < \gamma \;.\end{equation}
Possibly changing the sub-limits $v_j$, and extracting a subsequence of $n_s$,
we may assume that $v_{j, n_s}\to v_j$ as $s\to +\infty$, for all $1\leq j\leq k$.  Pick any $j$ such that $\langle v,v_j\rangle\neq 0$.
Since the vectors $\An{n_s}(x)\,v_{j,n_s}$ are pairwise orthogonal,
\begin{align*}
\norm{\An{n_s}(x)\,v}^2   & =  \sum_{j=1}^k \langle v, v_{j,n_s}\rangle^2\,\norm{\An{n_s}(x)\,v_{j,n_s}}^2 \\
&= \sum_{j=1}^k \langle v, v_{j, n_s}\rangle^2\, s_j( \An{n_s}(x))^2 
\geq \langle v, v_{j, n_s}\rangle^2\, s_j( \An{n_s}(x))^2 \;.
\end{align*} 
Hence, taking logarithms, dividing by $n_s$ and passing to the limit
we get
$$ \lim_{s\to +\infty}
\frac{1}{n_s}\,\log \norm {\An{n_s}(x)}\geq
\lim_{s\to +\infty}
\frac{1}{n_s}\,\log s_j( \An{n_s}(x)) = L_j(A)=\gamma\;, 
$$
which contradicts ~\eqref{absurd assumption}. This proves that
$$ \lambda_A(x,v) = \lim_{n\to+\infty}\frac{1}{n}\,\log\norm{\An{n}(x)\,v}= \gamma \;, $$
which concludes the proof. 
\end{proof}

%
%
%
%
%
%
%
%

\begin{corollary} 
\label{coro:lim mostexp infty = L1}
Given a cocycle $A$ such that $L_1(A)>L_2(A)$,
 for any $\mu$-regular point $x\in X$, and all
 $v\in \medir{\infty}(A)(x)\setminus\{0\}$, \,
$ \lambda_A(x,v)= \lambda_A^-(x,v)
=L_1(A)$.
\end{corollary}

\begin{proof}
Follows from proposition~\ref{prop:lim mostexp infty = Lk} with $k=1$.

\end{proof}

\begin{definition}\label{def adjoint cocycle}
The adjoint of a cocycle $(T,A)$ is the map 
$F_{A^\ast}:\Bundle\to\Bundle$, defined by \, $F_{A^\ast}(x,v)=(T^{-1} x, A(T^{-1} x)^\ast\,v)$.
This cocycle is denoted by  $(T^{-1},A^\ast)$, or simply  $A^\ast$.
\end{definition}

\begin{remark}\label{remark adjoint}\normalfont
 The adjoint cocycle satisfies for any $n\in\N$ and $x\in X$,
$$ (A^\ast)^{(n)}(x) = \An{n}(T^{-n} x)^\ast \;.$$
\end{remark}

\begin{proposition}
If $A$ is $\mu$-integrable then the adjoint cocycle $A^\ast$ is also  $\mu$-integrable.

Moreover, the cocycle $A$ and its adjoint $A^\ast$ have the same Lyapunov exponents, $L_i(A)=L_i(A^\ast)$ for all $i=1,\ldots, m$.

\end{proposition}

\begin{proof}
The integrability of $A^\ast$ 
follows from the relation $\norm{A}=\norm{A^\ast}$.

The second statement is a consequence of a linear operator and its adjoint sharing the same singular values. In fact, by proposition~\ref{prop: LE and SVs} and remark~\ref{remark adjoint}
\begin{align*}
L_i(A) &= \lim_{n\to+\infty} \frac{1}{n}\,\int_X \log \norm{s_i(\An{n}(x))}\,d\mu(x) \\
&= \lim_{n\to+\infty} \frac{1}{n}\,\int_X \log \norm{s_i(\An{n}(T^{-n} x))}\,d\mu(x) \\
&= \lim_{n\to+\infty} \frac{1}{n}\,\int_X \log \norm{s_i((A^\ast)^{(n)}(x))}\,d\mu(x) = L_i(A^\ast)\;. \quad 
\end{align*}
\end{proof}

\begin{lemma}\label{lemma  alpha v* v >0}
If $L_1(A)>L_{2}(A)$ then  for $\mu$-almost every  $x\in X$,
$$\aangle\left( \medir{\infty}(A^\ast)(x), \, \medir{\infty}(A)(x)\right)>0\;.$$
\end{lemma}

\begin{proof}
Take $0<\epsilon \ll \varkappa:= L_1(A)-L_2(A)$, and 
consider the measurable set $\Omega_0$ and the order $n_0\in\N$ provided by proposition~\ref{prop recurrent doubling sequence and AP estimates}. For $x\in\Omega_{0}$ let $\{m_i\}_i$ be an $\epsilon$-doubling sequence of avalanche times. Then for all $i\geq 0$,
\begin{align*}
 \aangle\left( \An{m_i}(x),\, \An{m_{i+1}-m_i}(T^{m_i} x) \right) & \asymp \frac{\norm{\An{m_{i+1}}(x)}}{\norm{\An{m_i}(x)}\,\norm{\An{m_{i+1}-m_i}(T^{m_i} x)}} \\
& \geq  e^{-4\,m_i\,\epsilon} \;. 
\end{align*}
Let 
$\Omega_i :=T^{m_i} \Omega_{0}$.
We have $\mu(X\setminus \Omega_i)<\epsilon$, and for all
$i\geq 0$ and $x\in\Omega_i$,
\begin{align*}
& \aangle\left( \medir{m_i}(A^\ast)(x),\, \medir{m_{i+1}-m_i}(A)(x) \right)   \\
&\qquad  \quad = \aangle\left( \mostexp(\An{m_i}(T^{-m_i} x)^\ast),\, \mostexp(\An{m_{i+1}-m_i}(x))  \right)\\
&  \qquad \quad = \aangle\left( \An{m_i}(T^{-m_i} x),\, \An{m_{i+1}-m_i}(x) \right)  
\gtrsim e^{-4\,m_i\,\epsilon} \;. 
\end{align*} 
Notice that by proposition~\ref{med:aec}, for $i$ large,
the distances
$$d(\medir{m_i}(A^\ast)(x),\medir{\infty}(A^\ast)(x))\quad \text{ and } \quad d(\medir{m_{i+1}-m_i}(A)(x),\medir{\infty}(A)(x))$$  
are much smaller than $e^{-4\,m_i\,\epsilon}$.
Hence
$$ \aangle\left( \medir{\infty}(A^\ast)(x),\, \medir{\infty}(A)(x) \right) 
\gtrsim e^{-4\,m_i\,\epsilon} >0  $$
on the set $\Omega_i$, which has measure $\mu(\Omega_i)>1-\epsilon$.
This proves the lemma.

\end{proof}

Given a measurable sub-bundle $\hatv:X\to \Pp(\R^m)$,
we call {\em unit measurable section} of $\hatv:X\to \Pp(\R^m)$ to any measurable function $v:X\to\R^m$ such that $\norm{v(x)}=1$ and $v(x)\in \hatv(x)$ for $\mu$-a.e. $x\in X$.

\begin{lemma}\label{A(x)*v(T x) not 0}
Assume $L_1(A)>L_2(A)$ and let $v:X\to\R^m$ be a unit measurable section of $\medir{\infty}(A)$.
Then \, $A(x)^\ast v(T x)\neq 0$ \, for $\mu$-almost every $x\in X$.
\end{lemma}

\begin{proof}
Let $v^\ast, v^\ast_n:X\to\R^m$ be  unit measurable sections of $\medir{\infty}(A^\ast)$ and $\medir{n}(A^\ast)$, respectively. By lemma~\ref{lemma  alpha v* v >0},
$\aangle(\hatv(T x),\hatv^\ast(T x))>0$ for $\mu$-a.e. $x\in X$.
By proposition~\ref{med:aec} applied to the adjoint cocycle  $A^\ast$, for $\mu$-a.e. $x\in X$ and all large enough $n \geq 1$,
\begin{align*}
 \alpha_0  &:=\aangle(\hatv(T x),\mostexp(\An{n}(T^{-n+1} (x)^\ast) ) =
\aangle(\hatv(T x),\mostexp( (A^\ast)^{(n)}(T x)) ) \\
& =  \aangle(\hatv(T x),\medir{n}( A^\ast)(T x))   = \aangle(\hatv(T x),\hatv_n^\ast(T x))>0\;. 
\end{align*}
Hence by item (a) of Proposition 2.13 in~\cite{LEbook-chap2},
$$ \norm{\An{n}(T^{-n+1} (x)^\ast\,v(T x)} \geq
\alpha_0 \, \norm{\An{n}(T^{-n+1} (x)^\ast}>0\;. $$
Finally, since
$$ \An{n}(T^{-n+1} (x)^\ast\,v(T x) =
\An{n-1}(T^{-n+1} (x)^\ast\,A(x)^\ast\, v(T x) \;, $$
we  infer that $A(x)^\ast\, v(T x)\neq 0$.

\end{proof}

From now on, given a matrix $A(x)$, and a projective, or Grassmannian, point $\hatv$
we will abbreviate $\varphi_{A(x)} \hatv$ writing   $A(x)\,\hatv$. The following proposition establishes the invariance of the most expanding sub-bundles $\medir{\infty}_k(A)$.

\begin{proposition}\label{ medir k invariance}
If $L_k(A)>L_{k+1}(A)$ then for $\mu$-a.e. $x\in X$,
\begin{enumerate}
\item[(a)] \quad $\displaystyle  A(x)^\ast \, [ \medir{\infty}_k(A)(T x)] = \medir{\infty}_k(A)(x) $,

\item[(b)] \quad $\displaystyle  A(x)^{-1}\, [\medir{\infty}_k(A)(T x)^\perp] = 
 \medir{\infty}_k(A)(x)^\perp$\,.
\end{enumerate}
\end{proposition}

\begin{proof}
 By Proposition 2.18 in~\cite{LEbook-chap2}, (b) reduces to (a).
 Working with exterior powers we can reduce (a) to the case $k=1$.

Let us abbreviate
$\hatv(x):=\medir{\infty}(A)(x)$,
$\hatv_n(x):=\medir{n}(A)(x)$,
$\hatv^\ast(x):=\medir{\infty}(A^\ast)(x)$ and 
$\hatv_n^\ast(x):=\medir{n}(A^\ast)(x)$.
With this notation, (a) reduces to  the identity
$A(x)^\ast \hatv(T x) = \hatv(x)$.

By proposition~\ref{med:aec},
\begin{align*}
\hatv(x) & \approx \hatv_n(x) = \mostexp(\An{n}(x)) =
\An{n}(x)^\ast\,\mostexp(\An{n}(x)^\ast) \\
&= \An{n}(x)^\ast\,\hatv_n^\ast(T^n x)
= A(x)^\ast\,\An{n-1}(T x)^\ast\,\hatv_n^\ast(T^n x)\;,
\end{align*}   
and analogously
$$ \hatv(T x)  \approx \hatv_{n-1}(T x) = \An{n-1}(T x)^\ast\,\hatv_{n-1}^\ast(T^n x) \;. $$
Hence
\begin{align*}
A(x)^\ast \,\hatv(T x) & \approx  A(x)^\ast \,\hatv_{n-1}(T x)
= \An{n}(x)^\ast\,\hatv_{n-1}^\ast(T^n x)\\
&\approx \An{n}(x)^\ast\,\hatv_n^\ast(T^n x)
=  \hatv_n(x) \approx \hatv(x)\;.
\end{align*} 
Item (a) follows from taking limits in these proximity relations.

On the first occurrence of $\approx$ we use the continuity of the action of $A(x)^\ast$ on the projective space, and lemma~\ref{A(x)*v(T x) not 0}, which asserts that
$A(x)^\ast v(T x)\neq 0$ for any unit measurable section
$v$ of $\hatv$.

On the second occurrence of $\approx$, take $0<\varkappa< L_1(A)-L_2(A)$, $0<\epsilon\ll \varkappa$ arbitrary small and, by Egorov's theorem,  a measurable subset $E\subset X$ such that $\hatv^\ast_n$ converges uniformly to $\hatv^\ast$ on $E$. Then choose a sequence of times $n\in\N$ such that $T^n x\in E$ and $\rgap(\An{n}(x)^\ast) = \rgap(\An{n}(x)) \geq e^{n\, \varkappa}$.
Because of this large gap ratio, $\An{n}(x)^\ast$ acts as  a strong contraction in a neighborhood of $\hatv_n^\ast(T^n x)$. But for  $T^n x\in E$,
$\hatv_n^\ast( T^n x)$ and $\hatv_{n-1}^\ast( T^n x)$ are both very close to $\hatv^\ast( T^n x)$, and hence close to each other.
Thus
$$  \delta(\,A(x)^\ast \,\hatv_{n-1}(T x), \, A(x)^\ast \,\hatv_{n}(T x)\, ) \ll 
\delta( \hatv_{n-1}(T x), \,  \hatv_{n}(T x)  )  $$
converges to $0$ as $n\to +\infty$.

On the last occurrence of $\approx$ we apply proposition~\ref{med:aec}.

\end{proof}

\begin{lemma}
\label{lemma 1/n f(Tn x)}
Given a measurable function $f:X\to \R$ such that $f-f\circ T\in L^1(X,\mu)$,
then for $\mu$-a.e. $x\in X$,
$$\lim_{n\to+\infty}\frac{1}{n}\,f(T^n x)= 0\;. $$
\end{lemma}

\begin{proof}
Note that
$\displaystyle \frac{1}{n}\, f(T^n x)
=  \frac{1}{n}\,f(x) - \frac{1}{n}\,\sum_{j=0}^{n-1} (f - f\circ T)(T^j x) \;,$
and conclude using Birkhoff's theorem.
\end{proof}

\begin{lemma}
\label{lemma non-integrable Birkhoff}
Let $T:X\to X$ be an ergodic m.p.t. on a probability space $(X,\mu)$, and $f:X\to (0,+\infty)$ a measurable non-integrable function. Then for $\mu$-a.e. $x\in X$,
$$ \lim_{n\to+\infty} \frac{1}{n}\,\sum_{j=0}^{n-1} f(T^j x) = +\infty\;. $$
\end{lemma}

\begin{proof}
Defining
$f_n= \max\{f,n\}$, by Lebesgue's monotone convergence theorem
$$ \lim_{n\to +\infty} \int_X f_n \,d\mu = 
\int_X f \,d\mu = +\infty \;.$$
For each $n\in\N$, since $f_n$ is $\mu$-integrable there is a full measure set $\mathscr{B}_n\subseteq X$ such that for all $x\in\mathscr{B}_n$,  $ \lim_{m\to+\infty} \frac{1}{m}\,\sum_{j=0}^{m-1} f_n(T^j x) = \int_X f_n \,d\mu$. Thus $\mathscr{B}= \cap_{n\in\N}\mathscr{B}_n$
is also a full measure set.

Given $x\in\mathscr{B}$ and  $L>0$, consider $p\in\N$ such that   $\int_X f_{p} \,d\mu > L$.
Since $x\in\mathscr{B}_p$, there is an order $n_0=n_0(x)>p$ such that for $n\geq n_0$
$$\frac{1}{n}\,\sum_{j=0}^{n-1} f(T^j x)\geq
\frac{1}{n}\,\sum_{j=0}^{n-1} f_{p}(T^j x)\geq L\;,$$
which proves the lemma.

\end{proof}

\begin{proposition}\label{prop  lim 1/n log ||An v*|| = L1}
Assume  $L_1(A)>L_{2}(A)$ and let $v,v^\ast:X\to \R^m$ be  unit measurable sections of $\medir{\infty}(A)$ and $\medir{\infty}(A^\ast)$, respectively.
Then  the functions  $\log \norm{ A  \,v^\ast  }$ and $\log \norm{ (A\circ T^{-1})^\ast  \,v }$ are $\mu$-integrable, and
$$ \int_X \log \norm{ A(x)\,v^\ast(x) }\, d\mu(x)= \int_X \log \norm{ A(T^{-1} x)^\ast\,v(x) }\, d\mu(x) = L_1(A)\;. $$
\end{proposition}

\begin{proof}
Because the  cocycles $A$ and $A^\ast$ play symmetric roles, it is enough proving the $\mu$-integrability of the function $\log \norm{ A  \,v^\ast  }$.

Applying proposition~\ref{ medir k invariance} to   $A^\ast$, we see that $A(x)\,v^\ast(x)=\pm v^\ast( T x)$. From this invariance relation, we get
for $\mu$-a.e. $x\in X$
$$ \log \norm{\An{n}(x)\,v^\ast(x)} =
\sum_{j=0}^{n-1} \log \norm{A(T^j x)\,v^\ast(T^j x)} \;. $$

Let $v_n:X\to \R^m$ be a unit measurable section 
 of $\medir{n}(A)$. For notational simplicity we will also
 write $\hatv^\ast(x):=\medir{\infty}(A^\ast)(x)$ and
  $\hatv_n(x):=\medir{n}(A)(x)$.
By item (a) of 
Proposition 2.13 in~\cite{LEbook-chap2},
$$\norm{\An{n}(x) v^\ast(x)} \geq \aangle\left( \hatv^\ast (x), \, \hatv_n(x) \right)\, \norm{\An{n}(x) }\;, $$
and hence
\begin{align*}
& \frac{1}{n}\,\log \norm{\An{n}(x)} +
\frac{1}{n}\,\log \aangle\left( \hatv^\ast (x), \, \hatv_n(x) \right) \\
& \qquad  \leq \frac{1}{n}\,\log \norm{\An{n}(x) v^\ast(x)}
 \leq \frac{1}{n}\,\log \norm{\An{n}(x)} \;.
\end{align*}

By proposition~\ref{prop  L1 = lim log An},
$\frac{1}{n}\,\log \norm{\An{n}(x)}$ converges to $L_1(A)$ almost surely. By lemma~\ref{lemma  alpha v* v >0},
$\aangle\left( \hatv^\ast (x), \, \hatv(x) \right)>0$,
and hence
$\frac{1}{n}\,\log \aangle ( \hatv^\ast (x), \, \hatv_n(x) )$ converges to zero.

Thus, for $\mu$-almost every $x\in X$,
$$ \lim_{n\to \infty} \frac{1}{n}\,
\sum_{j=0}^{n-1} \log \norm{A(T^j x)\,v^\ast(T^j x)} = 
\lim_{n\to \infty} \frac{1}{n}\,\log \norm{\An{n}(x) v^\ast(x)} = L_1(A)\;.$$

The function $\log \norm{A(x) v^\ast(x)}$ is bounded from above by the $\mu$-integrable function $\log^+ \norm{A(x)}$.
Hence, $h(x):= \log^+ \norm{A(x)} - \log \norm{A(x) v^\ast(x)}$
is a non-negative measurable function whose Birkhoff averages converge  $\mu$-almost everywhere to
$\int \log^+ \norm{A}\,d\mu - L_1(A)$.
By lemma~\ref{lemma non-integrable Birkhoff} it follows that
$h\in L^1(X,\mu)$, which implies that
$\log\norm{A\,v^\ast}\in L^1(X,\mu)$.

Thus, by Birkhoff's theorem,\,
$\int_X \log \norm{ A(x)\,v^\ast }\, d\mu = L_1(A)$.

\end{proof}

\begin{proposition}\label{prop 1/n log aaangle = 0}
Assume $L_1(A)>L_{2}(A)$. Then for $\mu$-a.e. $x\in X$,
\begin{enumerate}
\item[(a)]\quad $\displaystyle  \lim_{n\to+\infty} \frac{1}{n}\,\log \aangle\left( \medir{\infty}(A^\ast)(T^n x),\, 
 \medir{\infty}(A)(T^n x)
\right) =0$. 

\item[(b)]\quad $\displaystyle \limsup_{n\to+\infty} \frac{1}{n}\log \aangle \left( 
\An{n}(x) \medir{\infty} (A)(x), \medir{\infty} (A)(T^n x)\right)=0$.
\end{enumerate}

\end{proposition}

\begin{proof}
Take unit measurable sections $v,v^\ast:X\to \Proj$
of $\medir{\infty}(A)$ and $\medir{\infty}(A^\ast)$, respectively,
and as before let us write $\hatv(x):= \medir{\infty}(A)(x)$
and $\hatv^\ast(x):= \medir{\infty}(A^\ast)(x)$.

Consider the function $f(x):= \log \aangle ( \hatv^\ast(x),\, 
 \hatv(x) )$. By lemma~\ref{lemma 1/n f(Tn x)}, for (a) it is enough to prove that $f-f\circ T\in L^1(\mu)$.

By proposition~\ref{ medir k invariance} we have
\begin{align*}
f(x)-f(T x) &= \log \frac{ \aangle(\hatv^\ast(x),\, 
 \hatv(x) ) }{ \aangle(\hatv^\ast(T x),\, 
 \hatv(T x)) } = \log \frac{ \aangle(\hatv^\ast(x),\, 
 A(x)^\ast \hatv(T x) ) }{ \aangle(A(x) \,\hatv^\ast(x),\, 
 \hatv(T x)) } \\
 &= \log \frac{ \langle v^\ast(x),\, 
 A(x)^\ast v(T x) \rangle }{ \norm{A(x)^\ast v(T x)} }\,
 \frac{ \norm{A(x) v^\ast(x)} }{ \langle(A(x)  v^\ast(x),\, 
 v(T x)\rangle }  \\
  &= \log  \norm{A(x) v^\ast(x)} - \log \norm{A(x)^\ast v(T x)} \;. 
\end{align*}
By proposition~\ref{prop  lim 1/n log ||An v*|| = L1},
$\log  \norm{A\, v^\ast}\in L^1(X,\mu)$, and
$\log  \norm{A^\ast\, (v\circ T)}\in L^1(X,\mu)$.
Hence by lemma~\ref{lemma 1/n f(Tn x)}  this implies (a).

As before, we use the notation $\hatv_n$ and $\hatv_n^\ast$
for the sub-bundles $\medir{n}(A)$ and $\medir{n}(A^\ast)$, respectively. Since $\An{n}(x)\,\hatv_n(x)= \hatv_n^\ast(T^n x)$,
by Proposition 3.13 in \cite{LEbook-chap2} we have
\begin{align*}
\aangle( \An{n}(x)\,\hatv(x), \hatv(T^n x)) &\geq
\aangle(  \hatv^\ast(T^n x), \hatv(T^n x)) -
\delta( \hatv^\ast(T^n x), \hatv_n^\ast(T^n x) ) \\
& \qquad \qquad - 
\delta( \An{n}(x)\,\hatv_n(x), \An{n}(x)\,\hatv(x) ) \;.
\end{align*}

Now take $0<\varkappa <L_1(A)-L_2(A)$ and $0<\epsilon\ll \varkappa$ arbitrary small. By item (a), for all large enough $n$\, 
$\aangle(  \hatv^\ast(T^n x), \hatv(T^n x))\geq e^{-n\epsilon}$.
Because as $n$ grows, $\An{n}(x)$ has a large gap ratio,
it acts as  a strong contraction in a neighborhood of $\hatv_n^\ast(x)$. Hence by proposition~\ref{med:aec},
$$ \delta( \An{n}(x)\,\hatv_n(x), \An{n}(x)\,\hatv(x) )\ll
\delta(  \hatv_n(x),  \hatv(x) )\leq e^{-n\,(\varkappa-\epsilon)}\;. $$
We can not guarantee that the second distance 
$\delta( \hatv^\ast(T^n x), \hatv_n^\ast(T^n x) )$ converges to $0$ $\mu$-almost everywhere, but since $\hatv_n^\ast$ converges almost surely to $\hatv^\ast$, with the speed provided by proposition~\ref{med:aec}, for $\mu$-a.e. $x\in X$ there is a sequence of times $\{n_i\}_i$ such that
$$ \delta( \hatv^\ast(T^{n_i} x), \hatv_{n_i}^\ast(T^{n_i} x) )
\leq e^{-{n_i}\,(\varkappa-\epsilon)} \quad\forall\, i \;.$$
Thus, taking logarithms and dividing by $n$, (b) follows.

\end{proof}

\begin{proposition} \label{prop lambda wedge <= }
Given $x\in X$ and unit vectors  $v_k\in \wedge_k E(x)$ and
$v_r\in \wedge_r E(x)$,
\begin{align*}
 \lambda_{\wedge_{k+r} A}(x, v_k\wedge v_r) &\leq 
\lambda_{\wedge_{k} A}(x, v_k)  +
\lambda_{\wedge_{r} A}(x, v_r)  \;,\\
 \lambda_{\wedge_{k+r} A}^-(x, v_k\wedge v_r) &\leq 
\lambda_{\wedge_{k} A}^-(x, v_k)  +
\lambda_{\wedge_{r} A}^-(x, v_r)  \;.
\end{align*}

\end{proposition}

\begin{proof}
By item (a) of Proposition 3.12 in~\cite{LEbook-chap2},
$$ \log \norm{ \wedge_{k+r} \An{n}(x) v_k\wedge v_r}
\leq \log \norm{ \wedge_{k} \An{n}(x) v_k}
+ \log \norm{ \wedge_{r} \An{n}(x) v_r} \;. $$
Hence, dividing by $n$ an passing to the limit,
the inequalities follow.
\end{proof}

\begin{proposition} \label{prop lambda wedge >= }
Assume $L_k(A)>L_{k+1}(A)$.
Given unit vectors $v_k\in \wedge_k \left[ \medir{\infty}_k(A)(x)\right]$
and $v_r\in \wedge_r \left[ \medir{\infty}_k(A)(x)^\perp\right]$,
$$
\lambda_{\wedge_{k+r} A}^-(x, v_k\wedge v_r) =
\lambda_{\wedge_{k} A}(x, v_k)  +
\lambda_{\wedge_{r} A}^-(x, v_r) \;.$$

Moreover, if $ \lambda_{\wedge_{r} A}^-(x, v_r) =
\lambda_{\wedge_{r} A}(x, v_r)$ then 
$ \lambda_{\wedge_{k+r} A}^-(x, v_k\wedge v_r) =
\lambda_{\wedge_{k+r} A}(x, v_k\wedge v_r)$.
\end{proposition}

\begin{proof}
Because  $\wedge_k [\medir{\infty}_k(A)]$ has dimension one,  $v_k$ is a sub-limit of most expanding vectors for $\wedge_{k} \An{n}(x)$. Hence, by proposition~\ref{prop:lim mostexp infty = Lk}  we have
$\lambda_{\wedge_k A}^-(x,v_k) = \lambda_{\wedge_k A}(x,v_k)$.

In view of proposition~\ref{prop lambda wedge <= }, it is enough  to prove  that
$$\lambda_{\wedge_{k} A}(x, v_k)  +
\lambda_{\wedge_{r} A}^-(x, v_r)\leq
\lambda_{\wedge_{k+r} A}^-(x, v_k\wedge v_r) \;.$$

By proposition~\ref{ medir k invariance} we have $\wedge_r  \An{n}(x)\,v_r\in \wedge_r [ \medir{\infty}_k(A)(x)^\perp]$.
Hence by item (b) of Proposition 3.12 in~\cite{LEbook-chap2},
$$  \norm{ \wedge_k \An{n}(x) v_k }\, \norm{  \wedge_{r} \An{n}(x) \, v_r   }  \leq \alpha_n(x)^{-1} \,
\norm{  \wedge_{k+r} \An{n}(x) \, (v_k \wedge v_{r})   }\;,  $$
where
\begin{align*}
\alpha_n(x) & := \aangle_k\left( 
\An{n}(x) \medir{\infty}_k(A)(x), \medir{\infty}_k(A)(T^n x)\right)\\
&= \aangle\left( 
\wedge_k \An{n}(x) \medir{\infty}(\wedge_k A)(x), \medir{\infty}(\wedge_k A)(T^n x)\right) \;.
\end{align*}

Therefore, by  proposition~\ref{prop 1/n log aaangle = 0} (b),
\begin{align*}
& \lambda_{\wedge_{k} A}(x, v_k)  +
\lambda_{\wedge_{r} A}^-(x, v_r) = \\
&= \liminf_{n\to +\infty}\frac{1}{n}\,
\log \norm{ \wedge_k \An{n}(x) v_k }\, \norm{  \wedge_{r} \An{n}(x) \, v_{r}  }\\
&\leq  \liminf_{n\to +\infty}\frac{1}{n}\,
\log  \norm{  \wedge_{k+r} \An{n}(x) \, (v_k \wedge v_r)   } + \liminf_{n\to +\infty}\frac{1}{n}\,\log \alpha_n(x)^{-1}\\
&\leq   \lambda_{\wedge_{k+r} A}^-(x, v_k\wedge v_r)  - \limsup_{n\to +\infty}\frac{1}{n}\,\log \alpha_n(x)  = \lambda_{\wedge_{k+r} A}^-(x, v_k\wedge v_r)\; .
\end{align*}

Assume now that $\lambda_{\wedge_{r} A}^-(x, v_r)=\lambda_{\wedge_{r} A}(x, v_r)$.
Combining proposition~\ref{prop lambda wedge <= } with the previous inequality
\begin{align*}
\lambda_{\wedge_{k+r} A}(x, v_k\wedge v_r) &\leq
\lambda_{\wedge_{k} A}(x, v_k)  +  \lambda_{\wedge_{r} A}(x,v_r)\\
   &\leq  \lambda_{\wedge_{k} A} (x, v_k) + 
\lambda_{\wedge_{r} A}^- (x, v_r) \\
&\leq
\lambda_{\wedge_{k+r} A}^-(x, v_k\wedge v_r) \;, 
\end{align*}
which implies that $\lambda_{\wedge_{k+r} A}^-(x, v_k\wedge v_r)=\lambda_{\wedge_{k+r} A}(x, v_k\wedge v_r)$.

\end{proof}

\begin{definition}
Given a $\mu$-regular point $x\in X$ of a cocycle $A$, we call {\em limit singular basis} of the fiber $E(x)$ to any orthonormal basis $\{u_1,\ldots, u_m\}$ of $E(x)$ obtained as a sub-limit of a sequence of singular basis $\{u_{1,n},\ldots, u_{m,n}\}$ of $\An{n}(x)$.
\end{definition}

\begin{lemma}
\label{lemma limit sing basis}
Let $\{u_1,\ldots, u_m\}$ be a limit singular basis of $E(x)$
at some $\mu$-regular point $x\in X$. Then for all $i=1,\ldots, m$,
$$  \lambda_{\wedge_i A}^-(x, u_1\wedge \ldots \wedge u_i) = \lambda_{\wedge_i A}(x, u_1\wedge \ldots \wedge u_i) = L_1(\wedge_i A)\;. $$
\end{lemma}

\begin{proof}
Let  $\{u_{1,n},\ldots, u_{m,n}\}$ be a singular basis of $\An{n}(x)$, and $\{u_{1},\ldots, u_{m}\}$ a corresponding limit singular basis for the cocycle $A$. 
Choose $k$ such that
$$ L_1(\wedge_i A) = \ldots = L_{k}(\wedge_i A) 
> L_{k+1}(\wedge_i A)  \;.$$
Since $ u_1\wedge \ldots \wedge u_i$ is a sub-limit of
$ u_{1,n}\wedge \ldots \wedge u_{i,n}$, which is a sequence of vectors in  $\medir{n}_k(\wedge_i A)(x)$, we infer that
$ u_1\wedge \ldots \wedge u_i\in \medir{\infty}_k(\wedge_i A)(x)$. 
The conclusion follows by applying proposition~\ref{prop:lim mostexp infty = Lk}  to the cocycle $\wedge_i A$.
\end{proof}

\begin{proposition}
\label{prop Li (A vk perp) = Li+k(A)}
Consider  a cocycle $A$ such that $L_k(A)>L_{k+1}(A)$. Then 
$$  L_i\left( A\vert_{\mostexp_k^\perp} \right)= L_{i+k}(A) \quad \text{ for any  } \; 1\leq i \leq m-k,\; $$
where $A\vert_{\mostexp_k^\perp}$ stands for the restriction of  $A$ to the invariant  bundle $\medir{\infty}_k(A)^\perp$.
\end{proposition}

\begin{proof}
It is enough to see that
\begin{equation}\label{L1+L1(wedge A v perp)= L1 wedge}
 L_1(\wedge_k A)+ L_1\left( \wedge_i \, A\vert_{\mostexp_k^\perp} \right) 
= L_1(\wedge_{i+k} A)\;.
\end{equation}
In fact,  from~\eqref{L1+L1(wedge A v perp)= L1 wedge}, using corollary~\ref{L1 wedge = sum Lj}, 
$$
L_1(A)+ \dots + L_k(A) + L_1(A\vert_{\mostexp_k^\perp}) +\cdots +
L_i(A\vert_{\mostexp_k^\perp}) =  L_1(A) + \cdots + L_{i+k}(A)\;. 
$$
Therefore, the conclusion follows subtracting these identities for 
consecutive indexes $i$ and $i-1$.

Let us prove~\eqref{L1+L1(wedge A v perp)= L1 wedge}.
This identity is reduced to two inequalities. We will use propositions~\ref{prop lambda wedge <= } and~\ref{prop lambda wedge >= } 
 to establish each of these inequalities.

Fix a $\mu$-regular point $x\in X$, and consider a 
limit singular basis $\{ u_{1},\ldots, u_{m}\}$  of 
the fiber $E(x)$.
Hence by lemma~\ref{lemma limit sing basis},
\begin{align*}
L_1(\wedge_{i+k} A)  &= \lambda_{\wedge_{i+k} A}(x,u_1\wedge \ldots \wedge u_{k+i})   \\
& \leq \lambda_{\wedge_{k} A}(x, u_1\wedge \ldots \wedge u_{k}) + \lambda_{\wedge_{i} A}(x,u_{k+1}\wedge\ldots \wedge u_{k+i})\\
&\leq  L_1(\wedge_k A)  + 
L_1( \wedge_i \,A\vert_{\mostexp_k^\perp} )\; .
\end{align*}
On the last step we use that
$u_{k+1}\wedge\ldots \wedge u_{k+i}$ is a non zero vector in
the fiber of the bundle $\wedge_i [ \medir{\infty}_k(A)^\perp]$.

For the converse inequality, choose an orthonormal basis  $\{u_1,\ldots, u_k\}$ of $\medir{\infty}_k(A)(x)$ and extend it
with a limit singular basis $\{u_{k+1},\ldots, u_{m}\}$
for the cocycle  $A\vert_{\mostexp_k^\perp}$. 
By  lemma~\ref{lemma limit sing basis} applied to the cocycle  $A\vert_{\mostexp_k^\perp}$ we get
$$ \lambda_{\wedge_i A}^-(x, u_{k+1}\wedge \ldots \wedge u_{m}) = \lambda_{\wedge_i A} (x, u_{k+1}\wedge \ldots \wedge u_{m}) = L_1(\wedge_i A\vert_{\mostexp_k^\perp})\;. $$
Hence, by propositions~\ref{prop lambda wedge >= } 
and~\ref{prop:lim mostexp infty = Lk},
\begin{align*}
L_1(\wedge_{i+k} A) &\geq
 \lambda_{\wedge_{k+i} A}(x, u_1\wedge \ldots \wedge u_k \wedge u_{k+1}\wedge \ldots \wedge u_{k+i})\\
&=\lambda_{\wedge_k A}(x, u_1\wedge \ldots \wedge u_k ) +  \lambda_{\wedge_i \,A }(x, u_{k+1}\wedge \ldots \wedge u_{k+i})  \\
&= L_1(\wedge_k A) +   L_1( \wedge_i \,A\vert_{\mostexp_k^\perp} ) \;. 
\end{align*}
Together these two inequalities conclude the proof.
\end{proof}

\begin{proposition}
\label{prop medir (A vk perp) = medir(A)}
Consider integers $1\leq k <k+r \leq m$ such that 
$$L_{k}(A)>L_{k+1}(A) = \ldots = L_{k+r}(A)>L_{k+r+1}(A)\;. $$
Then for $\mu$-almost every $x\in X$,
$$ \medir{\infty}_{r}(A\vert_{ \mostexp_k^\perp})(x)
= \medir{\infty}_k(A)(x)^\perp\cap 
\medir{\infty}_{k+r}(A)(x) \;.$$

In particular, for every non-zero vector
$v$ in the fiber over $x$ of this sub-bundle,
$$ \lambda_A^-(x,v)= \lambda_A(x,v)
=\lim_{n\to+\infty} \frac{1}{n}\,\log\norm{\An{n}(x)\,v}=L_{k+1}(A)\;. $$
\end{proposition}

\begin{proof}
The stated relation is a simple application of 
Proposition 3.22 in~\cite{LEbook-chap2} to the
matrices $g=\An{n}(x)$. 
Notice that for a generic point $x\in X$ these matrices have exponentially small gap ratios
$\sgap_k(\An{n}(x))$ and $\sgap_{k+r}(\An{n}(x))$.
By that proposition
$$ \delta\left( 
\mostexp_r(\An{n}(x)\vert_{\mostexp_k^\perp}) , \,
\medir{n}_{k+r}(A)(x)\cap \medir{\infty}_k(A)(x)^\perp\, \right) \lesssim
\delta( \medir{n}_k(A)(x),\,  \medir{\infty}_k(A)(x)\,) $$
converges to zero. Hence the relation follows by taking the limit as $n$ tends to $+\infty$.

The last statement is a consequence of proposition~\ref{prop:lim mostexp infty = Lk}.
\end{proof}

\bigskip

Given a signature $\tau=(\tau_1,\ldots, \tau_k)$,  we define a sequence of partial functions
$\mostexp^{(n)}_\tau(A)$ on $X$ taking values on $\FF_\tau(\R^m)$,
$$ \mostexp^{(n)}_\tau(A)(x):=
\left\{ \begin{array}{lll}
\mostexp_\tau(A^{(n)}(x)) & & \text{ if }  \rgap_\tau(A^{(n)}(x))>1 \\
\text{undefined} & & \text{otherwise}.
\end{array} \right. $$

\medskip

We say that the Lyapunov spectrum of a cocycle $A$ has a {\em $\tau$-gap pattern}  when
$$ L_{\tau_j}(A) >L_{\tau_{j}+1}(A),\quad \text{ for all }\; 1\leq j<k\;.$$
The size of these gaps is measured by
$$ \gap_\tau(A):= \min_{1\leq j <k} L_{\tau_j}(A)  - L_{\tau_{j}+1}(A) \;.$$

\medskip

If moreover
$$ L_{\ell}(A)=L_{\ell+1}(A),\quad \text{ for all }\; 
\ell\notin \{\tau_1,\ldots, \tau_k\} \;,$$
we will say that 
the Lyapunov spectrum of $A$ has  {\em exact gap pattern  $\tau$}.
In this case we write
$\lambda_j(A):= L_{\tau_j}(A)$, for $j=1,\ldots, k+1$.
These   numbers span the complete Lyapunov spectrum of $A$ without   repetitions,
$$ \lambda_1(A) > \lambda_2(A) >\ldots > \lambda_k(A) > \lambda_{k+1}(A)\geq -\infty \;.$$

\begin{proposition}\label{mef:aec}
If the Lyapunov spectrum of $A$ has a $\tau$-gap pattern, then the sequence of partial functions $ \medir{n}_\tau(A)$ from $X$ to $\FF_\tau(\R^m)$ is almost everywhere Cauchy.
In particular, it  converges $\mu$ almost everywhere to a (total) measurable function $\medir{\infty}_\tau(A):X\to\FF_\tau(\R^m)$.
Moreover,  for $\mu$-a.e. $x\in X$,
$$\limsup_{n\to +\infty}\frac{1}{n} \,\log \, d(\medir{n}_\tau(A)(x), \medir{\infty}_\tau(A)(x)) \leq - \gap_\tau(A) <0 \;.$$
\end{proposition}

\begin{proof}
Apply proposition~\ref{prop: medir k exists}  at the dimensions $i=\tau_j$, with $j=1,\ldots, k$. 
\end{proof}

We are now able to state and to prove the  Oseledets Multiplicative Ergodic Theorem, which has two versions, one on the existence of the Oseledets filtration, and the other on the existence of the Oseledets decomposition.

\begin{theorem}[Oseledets I]
\label{Oseledets non invertible}
Let  $T:X\to X$   be an ergodic automorphism of a  probability space  $(X,\mathscr{A},\mu)$,
and let $F_A:\Bundle \to \Bundle$ be a $\mu$-integrable linear cocycle on a measurable bundle $\mathcal{B}\subseteq X\times\R^m$.

Then there exist $\lambda_1>\lambda_2>\ldots >\lambda_k\geq -\infty$
and a family of measurable functions 
$F_j:X\to \Gr(\R^m)$, $1\leq j\leq k$, such that for $\mu$-almost every  $x\in X$,
\begin{enumerate}
\item[(a)] $A(x)\,F_j(x)\subseteq F_j(T x)$ for $j=1,\ldots, k$

\item[(b)] $\{0\}= F_{k+1}(x) \subsetneq F_k(x)
\subsetneq \ldots \subsetneq F_{2}(x) \subsetneq F_1(x) = \mathcal{B}(x)$

\item[(c)] for every $v\in F_j(x)\setminus F_{j+1}(x)$,
\;
$\displaystyle 
\lim_{n\to +\infty} \frac{1}{n}\,\log \norm{\An{n}(x)\,v} = \lambda_{j} $.
\end{enumerate}
\end{theorem}

\begin{proof}
Assume the cocycle $A$ has a Lyapunov spectrum  with exact gap pattern $\tau=(\tau_1,\ldots, \tau_{k-1})$, where $0=\tau_0<\tau_1<\ldots <\tau_{k-1}<\tau_k = \dim E$, and $E=E(x)$ denotes the fiber of $\mathcal{B}$.
Set by convention $\medir{\infty}_{\tau_0}(A)=\{0\}$
and $\medir{\infty}_{\tau_{k}}(A)=E(x)$.

Define $F_j(x):= \medir {\infty}_{\tau_{j-1}}(A)(x)^\perp$
for $j=1,\ldots, k+1$, so that
$\dim F_j(x)= \dim E -\tau_{j-1}$. This implies (b).

The invariance (a) follows from proposition~\ref{ medir k invariance}.

To shorten notation let us write respectively
$\mostexp_k$, and $\mostexp_k^\perp$, instead of 
$\medir{\infty}_k(A)(x)$, and $\medir{\infty}_k(A)(x)^\perp$.
Given $v\in F_j\setminus F_{j+1}=\mostexp_{\tau_{j-1}}^\perp\setminus 
\mostexp_{\tau_{j}}^\perp$, consider the orthogonal decomposition
$v=u+v'$, with $u\in \mostexp_{\tau_{j-1}}^\perp \cap \mostexp_{\tau_{j}}$, $u\neq 0$, and $v'\in \mostexp_{\tau_{j}}^\perp$.
By proposition~\ref{prop medir (A vk perp) = medir(A)} the non-zero vector $u$ is in the fiber of
$$\medir{\infty}_{\tau_{j}-\tau_{j-1}}\left(
A\vert_{ \mostexp_{\tau_{j-1}}^\perp} \right)
= \medir{\infty}_{\tau_{j-1}}(A)^\perp \cap \medir{\infty}_{\tau_{j}}(A) \;, $$
and
\begin{align*}
\lambda_A^-(x,u)=\lambda_A(x,u)    =   L_{\tau_{j-1}+1}(A) = L_{\tau_{j}}(A)=\lambda_{j}(A)\;.
\end{align*}

Analogously, and using Proposition~\ref{prop Li (A vk perp) = Li+k(A)},
\begin{align*}
\lambda_A(x,v')   \leq  L_1\left( A\vert_{\mostexp_{\tau_{j}}^\perp}\right) = L_{\tau_{j}+1}(A) = L_{\tau_{j+1}}(A)=\lambda_{j+1}(A)<\lambda_{j}(A)\;.
\end{align*} 

Finally, applying item (d) of proposition~\ref{lambda LE properties} we infer that
$$ \lambda_A^-(x,v)= \lambda_A(x,v)= \lambda_A(x,u+v')=\lambda_A(x,u) = \lambda_{j}(A)\;.$$
This proves (c).

\end{proof}

%
%

\begin{definition}\label{def pseudo inverse}
Given a linear map $g:V\to V'$, between Euclidean spaces $V$ and $V'$, its pseudo inverse  $g^+:V'\to V$ is the composition $g^+:= (g\vert_{\Ker_g^\perp})^{-1}\circ \pi_{\Range_g}$
of the orthogonal projection
$\pi_{\Range_g}:V'\to \Range_g$ with the inverse of 
$g\vert_{\Ker_g^\perp}:\Ker_g^\perp\to\Range_g$.
\end{definition}

\begin{lemma}
For any linear map $g:V\to V'$,
and   integer $0\leq k\leq \dim V$,
$$\wedge_k(g^+)= (\wedge_k g)^+\;.$$
\end{lemma}

\begin{proof}
We make use of three properties which can be easily checked: 
(1) $\wedge_k \Range_g = \Range_{\wedge_k g}$,\\
(2) $\wedge_k(g\vert_E)=\wedge_k g\vert_{\wedge_k E}$ \; and \\ (3) $(\wedge_k g)^{-1}=\wedge_k (g^{-1})$.

Thus
$\wedge_k (\Ker_g^\perp ) = \wedge_k (\Range_{g^\ast} )
= \Range_{\wedge_k g^\ast} = \Ker_{\wedge_k g}^\perp$, and 
\begin{align*}
\wedge_k(g^+) &= \wedge_k  (g\vert_{\Ker_g^\perp})^{-1}\circ \wedge_k  \pi_{\Range_g} = 
(\wedge_k g\vert_{ \wedge_k \Ker_g^\perp})^{-1}\circ \pi_{\wedge_k \Range_g} \\
&= 
(\wedge_k g\vert_{ \Ker_{\wedge_k g}^\perp})^{-1}\circ \pi_{ \Range_{\wedge_k g}}  =(\wedge_k g)^+\;.
\end{align*}
Behind this cumbersome algebraic calculation there is a geometric meaning to grasp.
\end{proof}

\begin{definition}\label{def backward iterates}
Given a cocycle $A:X\to\Mat(m,\R)$, we define for $n>0$
$$ \An{-n}(x):= \An{n}(T^{-n} x)^+ \;. $$
\end{definition}

When the cocycle takes invertible values, i.e., $A:X\to \GL(m,\R)$,
the backward iterates $\An{-n}(x)$ correspond to forward iterates
by the inverse cocycle $(T^{-1},A^{-1})$.

\begin{theorem}[Oseledets II]\label{thm:oseledets2}
Let  $T:X\to X$   be an ergodic automorphism of a  probability space  $(X,\mathscr{A},\mu)$,
and let $F_A:\Bundle \to \Bundle$ be a $\mu$-integrable linear cocycle on a measurable bundle $\mathcal{B}\subseteq X\times\R^m$.

Then there exist $\lambda_1>\lambda_2>\ldots > \lambda_{k+1}\geq -\infty$
and a family of measurable functions 
$E_j:X\to \Gr(\R^m)$, $1\leq j\leq k+1$, such that for $\mu$-almost every $x\in X$,
\begin{enumerate}
\item[(a)] $\mathcal{B}(x)=\oplus_{j=1}^{k+1} E_j(x)$,
\item[(b)] $A(x)\,E_j(x) = E_j(T x)$ \, for $j=1,\ldots, k$, and \, $A(x)\,E_{k+1}(x)\subseteq E_{k+1}(T x)$,

\item[(c)]  for every $v\in E_j(x)\setminus \{0\}$,
\;  $\displaystyle  \lim_{n\to \pm \infty} \frac{1}{n}\,\log \norm{\An{n}(x)\,v} =
 \lambda_{j}  $,
 
 \item[(d)]  
 $\displaystyle  \lim_{n\to \pm \infty} \frac{1}{n}\,\log \abs{\sin \measuredangle_{{\rm min}} (\oplus_{j\leq l}E_j(T^n x), \oplus_{j>l}E_j(T^n x))} = 0$, \, for any $l=2,\dots, k$.
\end{enumerate} 
\end{theorem}

\begin{proof}
Assume that $A$ has a Lyapunov spectrum  with exact gap pattern $\tau=(\tau_1,\ldots, \tau_{k})$, where $0=\tau_0<\tau_1<\ldots <\tau_{k}<\tau_{k+1} = \dim E$, and $E=E(x)$ denotes the fiber of $\mathcal{B}$.
Set by convention $\medir{\infty}_{\tau_0}(A)=\{0\}$
and $\medir{\infty}_{\tau_{k+1}}(A)=E(x)$.

Define $E_j(x):= \medir{\infty}_{\tau_{j}}(A^\ast)(x)\cap \medir{\infty}_{\tau_{j-1}}(A)(x)^\perp$
for $j=1,\ldots, k+1$.

By proposition~\ref{ medir k invariance}, both sub-bundles
$\medir{\infty}_{\tau_{j}}(A^\ast)$ and $\medir{\infty}_{\tau_{j-1}}(A)^\perp$ are $A$-invariant, and hence the same is true about the intersection.
This proves (b).

For (a) consider the flag valued measurable functions
\begin{align*}
\mostexp_\tau^\ast(x) &=(\medir{\infty}_{\tau_1}(A^\ast)(x),\ldots, \medir{\infty}_{\tau_k}(A^\ast)(x) )\in \FF_{\tau}(E(x))  \;, \\
\mostexp_\tau^\perp(x) &=(\medir{\infty}_{\tau_k}(A)(x)^\perp,\ldots, \medir{\infty}_{\tau_1}(A)(x)^\perp )\in \FF_{\tau^\perp}(E(x)) \;. 
\end{align*}

According to Definition 3.5 in~\cite{LEbook-chap2}
we have
$$ \{ E_j(x)\}_{1\leq j\leq k+1} = 
\mostexp_\tau^\ast(x) \sqcap  \mostexp_\tau^\perp(x) \;.  $$

Thus, in view of Proposition 3.14  in~\cite{LEbook-chap2}
it is now enough to see that
$\theta_{\sqcap}(\mostexp_\tau^\ast(x), \mostexp_\tau^\perp(x)) >0$ for $\mu$-a.e. $x\in X$.
But by Definition 3.4
and Lemma 3.10  in~\cite{LEbook-chap2},
\begin{align*}  \theta_{\sqcap}(\mostexp_\tau^\ast, \mostexp_\tau^\perp ) & =\min_{1\leq i\leq k}
\theta_\cap( \medir{\infty}_{\tau_i}(A^\ast),
\medir{\infty}_{\tau_i}(A)^\perp  ) \\
& =\min_{1\leq i\leq k}
\aangle_{\tau_i} ( \medir{\infty}_{\tau_i}(A^\ast),
\medir{\infty}_{\tau_i}(A)  ) \\
& =\min_{1\leq i\leq k}
\aangle  ( \medir{\infty}(\wedge_{\tau_i} A^\ast),
\medir{\infty}(\wedge_{\tau_i} A)   ) >0\;.
\end{align*}
The final positivity follows from lemma~\ref{lemma  alpha v* v >0}. This proves (a), or in other words  that $\{ E_j(x)\}_{1\leq j\leq k+1}$ is a direct sum decomposition of $E(x)$ with
$\dim E_j(x)= \tau_j-\tau_{j-1}$, for all $j=1,\ldots, k+1$.

We   prove (c) through several reductions.

Consider first the case $j=1$ and $\tau_1=1$.
In this case $\tau_0=0$ and the intersection sub-bundle
$E_j$  is the $1$-dimensional $A$-invariant bundle $\medir{\infty}(A^\ast)$.

Let $v^\ast:X\to\Proj$ be a unit measurable section of this bundle. By invariance we have   $A(x)\,v^\ast(x)=\pm v^\ast( T x)$ or $\mu$-a.e. $x\in X$,
and hence we get
$$ \log \norm{\An{n}(x)\,v^\ast(x)} =
\sum_{j=0}^{n-1} \log \norm{A(T^j x)\,v^\ast(T^j x)} \;. $$

By proposition~\ref{prop  lim 1/n log ||An v*|| = L1},
the function $\log \norm{A\,v^\ast}$ is $\mu$-integrable, with $\int \log \norm{A\,v^\ast}\,d\mu=L_1(A)$.
Therefore, by Birkhoff's ergodic theorem 
we have
$$ \lim_{n\to+\infty} \frac{1}{n}\,\log \norm{\An{n}(x)\,v^\ast(x)} = L_1(A)\;. $$
On the other hand, since $T$ is invertible, the Birkhoff averages
$$ \log \norm{\An{n}(T^{-n} x)\,v^\ast(T^{-n} x)} =
\sum_{j=0}^{n-1} \log \norm{A(T^{-j} x)\,v^\ast(T^{-j} x)}   $$
also converge $\mu$-almost everywhere  to $L_1(A)$.
Now, inverting the relation
$$  \An{n}(T^{-n} x)\,v^\ast(T^{-n} x) = \norm{\An{n}(T^{-n} x)\,v^\ast(T^{-n} x)}\, v^\ast(x) \;,$$ 
we get
$$  \An{n}(T^{-n} x)^+ \,v^\ast(x) =  \norm{\An{n}(T^{-n} x)\,v^\ast(T^{-n} x)}^{-1}\, v^\ast(T^{-n} x) \;,$$ 
so that 
$$ \log \norm{\An{-n}(x)\,v^\ast(x) }
= -\log \norm{\An{n}(T^{-n} x)\,v^\ast(T^{-n} x)}\;.$$
Thus
$$ \lim_{n\to  \infty} \frac{1}{n}\,\log \norm{\An{-n}(x)\,v^\ast(x)} = - L_1(A)\;. $$

Next consider the case $j=1$ and $r = \tau_1>1$.
In this case $\tau_0=0$ and the intersection sub-bundle $E_j$ is the $r$-dimensional $A$-invariant bundle $\medir{\infty}_r(A^\ast)$.

Given a unit vector $v_1$ in $\medir{\infty}_r(A^\ast)$,
include it in some orthonormal basis $\{v_1,\ldots, v_r\}$ of $\medir{\infty}_r(A^\ast)$ and take the unit $r$-vector
$w=v_1\wedge \ldots \wedge v_r$.
Applying the previous case to the cocycle $\wedge_r A$ and
$w\in \medir{\infty}(\wedge_r A)$, we conclude that
$$ \lim_{n\to  \pm \infty} \frac{1}{n}\,\log \norm{\wedge_r \An{n}(x)\,(v_1\wedge \ldots \wedge v_r)} =   L_1(\wedge_r A)=r\,L_1(A)\;. $$
By proposition~\ref{prop lambda wedge <= } we have
$$  \log\norm{\wedge_r \An{n}(x)(v_1\wedge\ldots \wedge v_r) } 
\leq \sum_{i=1}^r 
\log\norm{\An{n}(x)v_i} \leq
r\,\log\norm{\An{n}(x) }\;, $$
and since both upper and lower bounds of this sum converge to $r\,L_1(A)$, as $n\to\pm \infty$, we conclude that for all $i=1,\ldots, r$, and in particular for $i=1$,
$$\lim_{n\to\pm \infty}\frac{1}{n}\,\log \norm{\An{n}(x)v_i} =L_1(A)\;. $$

Finally consider the general case, where $2\leq j \leq k$. 
By proposition~\ref{prop medir (A vk perp) = medir(A)},
$$  E_j = \medir{\infty}_{\tau_{j}}(A^\ast)\cap \medir{\infty}_{\tau_{j-1}}(A)^\perp
=  \medir{\infty}_{\tau_j-\tau_{j-1}}(A\vert_{\mostexp_{\tau_{j-1}}^\perp} )\;.$$
We denote this $A$-invariant sub-bundle by $\mathcal{B}_j$.
Given a non-zero vector $v\in\mathcal{B}_j(x)$, applying the previous case to the restricted cocycle
$A\vert_{\mathcal{B}_j}$ by proposition~\ref{prop Li (A vk perp) = Li+k(A)}
\begin{align*}
\lim_{n\to\pm \infty}\frac{1}{n}\,\log \norm{\An{n}(x)\,v} &=
\lim_{n\to\pm \infty}\frac{1}{n}\,\log \norm{(A\vert_{\mathcal{B}_j})^{(n)}(x)\,v} \\
&= L_{\tau_j-\tau_{j-1}}(A\vert_{\mathcal{B}_j})
=L_{\tau_j}(A)=\lambda_j(A)\;.
\end{align*}
This proves (c).

For the last item, (d), notice that
$\oplus_{j\leq l} E_j = \medir{\infty}_{\tau_l}(A^\ast)$,
while
$\oplus_{j> l} E_j =\medir{\infty}_{\tau_l}(A)^\perp$.
Hence by item (d) of Proposition 2.10 in~\cite{LEbook-chap2},
\begin{align*}
\abs{\sin \measuredangle_{{\rm min}} &(\oplus_{j\leq l}E_j(x),  \oplus_{j>l}E_j(x))}  =
\deltamin( \medir{\infty}_{\tau_l}(A^\ast),
\medir{\infty}_{\tau_l}(A)^\perp 
)\\
&\geq 
\aangle_{\tau_l}( \medir{\infty}_{\tau_l}(A^\ast),
\medir{\infty}_{\tau_l}(A) 
) =\aangle( \medir{\infty}(\wedge_{\tau_l} A^\ast),
\medir{\infty}(\wedge_{\tau_l} A) 
)
\;.
\end{align*}
Thus item (d) follows by proposition~\ref{prop 1/n log aaangle = 0} (a).

\end{proof}

\pagebreak


\section{Abstract continuity theorem of the Oseledets filtration and decomposition}
\label{continuity_oseledets}
Given a space of cocycles  satisfying base and uniform fiber LDT estimates, we prove that the Oseledets filtration and decomposition vary continuously with the cocycle in an $L^1$ sense. 
We first prove the continuity of the most expanding direction, as it contains the main ingredients of our argument. We then define the space of measurable filtrations and endow it with an appropriate topology. Using the construction of the Oseledets filtration and decomposition in Subsection~\ref{met}, we deduce the continuity of these two quantities from that of the most expanding direction. This is obtained via some Grassmann geometrical considerations established in Chapter 2 of~\cite{LEbook}, see also~\cite{LEbook-chap2}.

We begin by describing the abstract setup of our continuity theorems.

\subsection{Assumptions on the space of cocycles}
\label{act assumptions}
\begin{definition}\label{def:cocycle:space}
A  space of measurable cocycles  $\mathcal{C}$
is any class of matrix valued functions
$A:X\to \Mat (m, \R)$, where $m\in\N$ is not fixed, such that
every  $A:X\to \Mat (m, \R)$ in $\mathcal{C}$ has the following properties:
\begin{enumerate}
\item $A$ is measurable.
\item $\norm{A}\in L^\infty(\mu)$.
\item The exterior powers $\wedge_k A: X\to \Mat_{\binom{m}{k}}(\R)$ are in $\mathcal{C}$,
for $k\leq m$.
\end{enumerate}

Each subspace 
$\mathcal{C}_m :=\{\, A\in \mathcal{C} \ | 
A:X\to \Mat_m(\R)\,\} $
is a-priori endowed with a distance 
$\dist \colon \mathcal{C}_m \times \mathcal{C}_m \to [0,+\infty)$ which is at least as fine as the $L^\infty$ distance, i.e. for all $A, B \in \cocycles_m$ we have
$$\dist (B, A) \ge \norm{B-A}_{L^\infty}\,.$$  

We assume a correlation between the distances on each of these subspaces, namely the map 
$\cocycles_m \ni A \mapsto \wedge_k \, A \in \cocycles_{m \choose k} $
is locally Lipschitz. 
\end{definition}

\smallskip

The functions $ \frac{1}{n} \, \log \norm{ \An{n} (x) }$ are integrable, and their integrals, denoted by $L^{(n)}_1 (A)$, will be referred to as finite scale (top) Lyapunov exponents. 

We  need stronger integrability assumptions on these  functions.
\begin{definition}\label{def:Lp-bounded}
A cocycle $A \in \cocycles$ is called $L^p$-bounded if there is a constant $C < \infty$, which we call its $L^p$-bound, such that for all $n \ge 1$ we have:
\begin{equation}\label{eq:Lp-bounded}
\bnorm{ \frac{1}{n} \, \log \norm{\An{n} (\cdot)}}_{L^p} < C
\end{equation}
\end{definition}
A cocycle $A \in \cocycles_m$ is called uniformly $L^p$-bounded, if the above bound holds uniformly near $A$.

\medskip

Given a cocycle $A \in \cocycles$ and an integer $N \in \N$, denote by $\mathscr{F}_N (A)$ the algebra generated by the sets $\{ x \in X \colon \norm{\An{n} (x)}  \le c \}$ or $\{ x \in X \colon \norm{\An{n} (x)}  \ge c \}$ where $c \ge 0$ and $0 \le n \le N$.

Let $\observables$ be a set of measurable functions $\xi \colon X \to \R$, which we call observables. Let $A \in \cocycles$.

\begin{definition}\label{compatibility_condition}
We say that $\observables$ and $A$ are compatible if for every integer $N \in \N$, for every set $F \in \mathscr{F}_N (A)$ and for every $\epsilon > 0$, there is an observable $\xi \in \observables$ such that:
\begin{equation}\label{compatibility_condition-eq}
\ind_{F} \le \xi \quad \text{and} \quad \int_X \xi \, d \mu \le \mu (F) + \epsilon\,.
\end{equation}
\end{definition}

\medskip


To describe the LDT estimates  we introduce the following formalism. From now on, $\devf, \, \mesf \colon (0, \infty) \to (0, \infty)$ will represent  functions that describe respectively, the size of the deviation from the mean and the measure of the deviation set. We assume that the deviation size functions $\devf (t)$ are non-increasing. We assume that the deviation set measure functions $\mesf (t)$  are continuous and strictly decreasing to $0$, as $t \to \infty$, at least like a power and at most like an exponential: for some $c_0 > 0$,
$$e^{- c_0 \, t} < \mesf (t) < t^{- 10} \ \text{  as } t \to \infty \,.$$ 

We  use the notation $\dev_n := \devf (n)$ and $\mes_n := \mesf (n)$ for integers $n$.

Let $\params$ be a set of triplets  $\param = (\nzerobar, \devf, \mesf)$, where $\nzerobar$ is an integer and $\devf$ and $\mesf$ are deviation functions.  We call the elements of $\params$ LDT parameters. 
 
 \medskip

We now define the base and fiber LDT estimates. 

\begin{definition}\label{def:base-LDT}
An observable $\xi \colon X \to \R$ satisfies a base-LDT estimate w.r.t. a space of parameters $\params$  if for every $\ep > 0$ there is $\param = \param (\xi, \ep) \in \params$, $\param = (\nzerobar, \devf, \mesf)$, such that for all $n \ge \nzerobar$ we have $\dev_n \le \ep$ and
\begin{equation}\label{eq:base-LDT}
\mu \, \{ x \in X \colon \abs{ \frac{1}{n} \, S_n\xi \,(x) - \avg{\xi } } > \dev_n \} < \mes_n\;. 
\end{equation} 
\end{definition}

 \medskip

\begin{definition}\label{def:fiber-LDT}
A measurable cocycle $A:X\to\Mat(m,\R)$ satisfies a fiber-LDT estimate w.r.t. a space of parameters $\params$  if  for every $\ep > 0$ there is $\param = \param (A, \ep) \in \params$, $\param = (\nzerobar, \devf, \mesf)$, such that for all $n \ge \nzerobar$ we have $\dev_n \le \ep$ and
\begin{equation}\label{eq:ufiber-LDT-o}
\mu \, \{ x \in X \colon \abs{ \frac{1}{n} \, \log \norm{ \An{n} (x) } - \Lan{n}{1} (A) } > \dev_n \} < \mes_n\;. 
\end{equation}

The cocycle $A$ satisfies a uniform fiber-LDT estimate if \eqref{eq:ufiber-LDT-o} holds in a neighborhood of $A$ with the same LDT parameter $\param$.
\end{definition}

\bigskip

Our  continuity results on the most expanding direction, the Oseledets filtration and decomposition hold under the same  assumptions in the abstract continuity theorem (ACT) of the Lyapunov exponents  (Theorem 3.1 in~\cite{LEbook} or Theorem 1.1 in~\cite{LEbook-chap3}). 

Thus for the rest of the paper, we will be under the following:
\subsection*{Assumptions.}\label{oseledets continuity assumptions} 
Consider an ergodic system $(X,  \mu, T)$, a space of measurable cocycles $\mathcal{C}$, a set of observables $\observables$, a set of LDT parameters $\params$ such that
\begin{enumerate}
\item $\observables$ is compatible with every cocycle $A \in \cocycles$.
\item Every observable $\xi \in \observables$ satisfies a base-LDT estimate.
\item Every cocycle with finite top LE is uniformly $L^p$-bounded, where $1 < p \le \infty$.  For simplicity of notations we let $p=2$.
\item Every cocycle $A \in \cocycles$ such that $L_1 (A) > L_2 (A)$ satisfies a uniform fiber-LDT estimate.
\end{enumerate}

\subsection{Continuity of the most expanding direction}
\label{direction}
We employ the  Lipschitz estimates on Grassmann manifolds and the avalanche principle derived in Chapter 2 of~\cite{LEbook}, see also~\cite{LEbook-chap2}.

Recall from Subsection~\ref{met} that the most expanding direction of the $n$-th iterate of a cocycle $A \in \cocycles$ defines a partial function 
$$
\medir{n} (A) (x) := \begin{cases}
\mostexp (\An{n} (x)) & \text{ if } \ \rgap (\An{n} (x)) > 1 \\
\text{undefined} & \text{ otherwise}.
\end{cases}
$$

By Proposition~\ref{med:aec}, as $n \to \infty$ the functions $\medir{n} (A) (x)$ converge $\mu$ a.e. to a measurable function $\medir{\infty} (A) \colon X \to \Proj$.

Let $L^1 (X, \Proj)$ be the space of all Borel measurable functions $\filt \colon X \to \Proj$.  Consider the distance 
$$
\dist (\filt_1, \filt_2) := \int_X d ( \filt_1 (x), \filt_2 (x) ) \, \mu (d x)\,,
$$
where the quantity under the integral sign refers to the distance between points in the projective space $\Proj$.

Clearly all the functions $\medir{n} (A)$ are in $L^1 (X, \Proj)$, and by dominated convergence we have that as $n \to \infty$, 
\begin{equation}\label{medir-L1conv}
\medir{n} (A) \to \medir{\infty} (A)  \quad \text{ in } \ L^1 (X, \Proj).
\end{equation}

We will prove that if $L_1 (A) > L_2 (A)$, then locally near $A$, the map $B \mapsto \medir{\infty} (B)$ is continuous with a modulus of continuity depending on the LDT parameter. 

We do so by deriving a {\em quantitative} version of the convergence in \eqref{medir-L1conv}, which moreover is somewhat uniform in phase and cocycle.  This more precise convergence comes as a consequence of the availability of the LDT estimates for our system, as the exceptional sets of phases in the domain of availability of the avalanche principle can be precisely (and uniformly in the cocycle) measured. 

\medskip

Fix a cocycle $A \in \cocycles$ such that $L_1 (A) > L_2 (A)$. Let $\gapc (A) := L_1 (A) - L_2 (A) > 0$ and $\ep_0 := \gapc (A) /100$. What follows is a  bookkeeping of various exceptional sets related to notions from Chapter 3 in~\cite{LEbook}, see also~\cite{LEbook-chap3}. They will eventually define and measure the exceptional sets in the domain of applicability of the avalanche principle (AP) in Proposition~\ref{AP-practical}, for certain sequences of iterates of a cocycle $B$ in a small neighborhood of $A$.

\medskip

Pick for the rest of this subsection $\delta = \delta (A) > 0$, $\nzerobar = \nzerobar(A) \in \N$, $\mesf = \mesf (A) \in \mesfs$ such that, by Lemma 5.1 in~\cite{LEbook-chap3} we have: for all $B \in \cocycles_m$ with $\dist (B, A) < \delta$,
\begin{equation} \label{1-most-dir}
 \rgap (\Bn{n} (x)) > e^{ n \, \gapc (A) /2} =: \frac{1}{\varkappa_n}  \ ( > 1)
\end{equation}
holds for all $n \ge \nzerobar$ and for all $x$ outside a set of measure $ < \mes_n$, and
\begin{equation} \label{2-most-dir}
\abs{ \Lan{n}{1} (B) - \Lan{m}{1} (A) } < \gapc (A) /20 
\end{equation}
holds for all $m, n \ge \nzerobar$.

As before, \eqref{1-most-dir} will ensure the gap condition in the AP, while \eqref{2-most-dir} via Lemma 4.2 in~\cite{LEbook-chap3} will ensure the angle condition.

\smallskip

Fix a cocycle $B$ with $\dist (B, A) < \delta$. We will define, for all scales $n \ge \nzerobar$, the exceptional sets outside which the AP can be applied for various block lengths and configurations of block components.

\smallskip

The exceptional set in the nearly uniform upper semicontinuity of the maximal Lyapunov exponent (Proposition 2.1 in~\cite{LEbook-chap3}) depends only on $A$, and we denote it by $\B_n^{usc} (A)$. Its measure is $\mu \, [  \B_n^{usc} (A) ] < \mes_n$.

Let
$$\B_n^{ldt} (B) := \{ x \in X \colon \abs{ \frac{1}{n} \, \log \norm{ \Bn{n} (x)  } - \Lan{n}{1} (B)  } > \ep_n \}$$
be the exceptional set in the uniform fiber-LDT estimate. Its measure is $\mu \, [  \B_n^{ldt} (B) ] < \mes_n$.

Let
$$\B_n^{g} (B) := \B_n^{ldt} (B) \cup \B_n^{usc} (\wedge_2 A)\,. $$

A simple inspection of the proof of Lemma 2.2 in~\cite{LEbook-chap3} shows that $\B_n^{g} (B)$ is the exceptional set in \eqref{1-most-dir}, and its  measure satisfies $\mu \, [  \B_n^{g} (A) ] \less \mes_n$. Note also that \eqref{1-most-dir} ensures that $ \mostexp ( \Bn{n} (x) )$ is defined (since there is a gap between the two largest singular values). 

Moreover, a simple inspection of the proof of Lemma 4.2 in~\cite{LEbook-chap3}, combined with \eqref{2-most-dir}, shows that for $2 n \ge m_1, m_2 \ge n \ge \nzerobar$ the bound
\begin{equation} \label{3-most-dir}
\frac{\norm{\Bn{m_2+m_1} (x)}}{\norm{\Bn{m_2} (\transl^{m_1} x)} \cdot \norm{\Bn{m_1} (x)}} > e^{- (m_1+m_2) \, \gapc (A)/20} > e^{- n \, \gapc (A)/5} =: \varepsilon_n
\end{equation}
holds provided that
$$x \notin  \B_{m_2+m_1}^{ldt} (B) \cup  \B_{m_1}^{ldt} (B) \cup T^{- m_1} \, \B_{m_2}^{ldt} (B)\,.$$

Note that from  \eqref{1-most-dir} and \eqref{3-most-dir} that 
$\frac{\varkappa_n}{\varepsilon_n^2} = e^{- n \, \gapc(A)/10} < \mes_n  \ll 1\,,$
hence  the condition on 
$\varkappa$ and $\varepsilon$ from the AP is satisfied.

The bound on the distance between most expanding directions in the conclusion of the AP  is 
\begin{equation}\label{kappa over ep}
\frac{\varkappa_n}{\varepsilon_n} = e^{- 3 \, n \, \gapc(A) / 10} < \mes_n\,.
\end{equation}

When using the AP,  we will always apply \eqref{3-most-dir} to configurations for which $n$ is fixed and $m_1 = n$ while $n \le m_2 \le 2 n$. This motivates defining
$$\B_n^a (B) := \bigcup_{n\le m \le 3 n} \, [ \B_m^{ldt} (B) \cup T^{- n} \,  \B_m^{ldt} (B)]\,.$$
Clearly $\mu \, [  \B_n^{a} (B) ] \less n \,  \mes_n$, and if $x \notin \B_n^a (B)$, then the angle condition will be ensured for block components of the kind indicated above.

Let
$$\B_n^{ga} (B) := \B_n^g (B)  \cup \B_n^a (B)\,, $$
so $\mu \, [  \B_n^{ga} (B) ] \less n \,  \mes_n$, and if $x \notin \B_n^{ga} (B)$, both the gap and the angle conditions hold for appropriate  block components at scale $n$.

Let $n_0 \ge \nzerobar$ and $2 n_0 \le n_1 \le \mes_{n_0}^{-1/2}$. If we define
\begin{equation}\label{bad-set-ap-o}
\B_{n_0}^{ap} (B):= \bigcup_{0 \le i < n_0^{-1} \, \mes_{n_0}^{-1/2}} \, T^{- i n_0} \, \B_{n_0}^{ga} (B)\,,
\end{equation}
then $\mu \, [  \B_{n_0}^{ap} (B) ] \less n_0^{-1} \, \mes_{n_0}^{-1/2} \, n_0 \, \mes_{n_0} < \mes_{n_0}^{1/2}$ and if $x \notin \B_{n_0}^{ap} (B)$, then the AP can be applied to a block of length $n_1$ whose $n$ components have lengths $n_0$, except for the last, whose length is the remaining integer $m$ satisfying $n_0 \le m \le 2 n_0$. 

Now define for all $n \ge \nzerobar$ the nested, decreasing sequence of exceptional sets
\begin{equation}\label{bad-set-apn-o}
\B_{n}^\flat (B) := \bigcup_{k \ge n} \, \B_{k}^{ap} (B)\,.
\end{equation}
Clearly
$$ \mu \, [  \B_{n}^{\flat} (B) ]  \le \sum_{k \ge n} \, \mu \, [  \B_{k}^{ap} (B) ]  \less  \sum_{k \ge n} \, \mes_k^{1/2} \less \mes_n^{1/2}$$
and if $ x \notin  \B_{n}^{\flat} (B) $ then for {\em any} scales $n_0, n_1$ such that $n_0 \ge n$ and $2 n_0 \le n_1 \le \mes_{n_0}^{-1/2}$, since $x \notin \B_{n_0}^{ap} (B)$, the AP can be applied to a block of length $n_1$ whose components have lengths $\asymp n_0$.

\begin{remark}\label{bad-set-summary-o}\normalfont
Let us summarize the accounting above. Given a cocycle $A \in \cocycles_m$ with $L_1 (A) > L_2 (A)$, there are parameters $\delta, \nzerobar, \mesf$ depending only on $A$ so that for any cocycle $B \in \cocycles_m$ with $\dist (B, A) < \delta$ and for any scale $k \ge \nzerobar$, there are exceptional sets $\B_{k}^{ap} (B)$ and $ \B_{k}^{\flat} (B) $ of measure $< \mes_k^{1/2}$ such that
\begin{enumerate}
\item  If $n_0 \ge \nzerobar$, $2 n_0 \le  n_1 \le \mes_{n_0}^{-1/2}$ and  if $x \notin \B_{n_0}^{ap} (B)$, then the AP with parameters $\varkappa_{n_0}, \varepsilon_{n_0}$ can be applied to a block $\Bn{n_1} (x)$ of length $n_1$ whose components have lengths $n_0$, except possibly for the last, whose length is between $n_0$ and $2 n_0$. 

\item If $n \ge \nzerobar$ and if $x \notin  \B_{n}^{\flat} (B)$, then for any scales $n_0, n_1$ such that $n_0 \ge n$ and $2 n_0 \le n_1 \le \mes_{n_0}^{-1/2}$,  the AP with parameters $\varkappa_{n_0}, \varepsilon_{n_0}$ can be applied to a block $\Bn{n_1} (x)$ of length $n_1$ whose components have lengths $\asymp n_0$.
\end{enumerate}
\end{remark}

The following results are now easy to phrase and to prove.

\begin{lemma} \label{lemma1-most-dir}
Let $n_1, n_0 \in \N$ such that $n_0 \ge \nzerobar$ and $2 n_0 \le n_1 \le \mes_{n_0}^{-1/2}$.
If $x \notin \B_{n_0}^{ap} (B)$ then
\begin{equation} \label{eq-l1-most-dir}
d ( \mostexp ( \Bn{n_1} (x) ), \, \mostexp ( \Bn{n_0} (x) ) ) <  \frac{\varkappa_{n_0}}{\varepsilon_{n_0}} \,.
\end{equation}
\end{lemma}

\begin{proof}
Consider the block $ \Bn{n_1} (x)$ and break it down into $n-1$ many blocks of length $n_0$ each, and a remaining block of length $m$ with $n_0 \le m < 2 n_0$. In other words, write $n_1 = (n-1) \, n_0 + m$, for some $n_0 \le m < 2 n_0$ and define
$$g_i = g_i (x) := \Bn{n_0} (\transl^{i \, n_0} \, x)$$
for $0 \le i \le n-2$, and
$$g_{n-1} = g_{n-1} (x) := \Bn{m} (\transl^{(n-1) \, n_0} \, x)\,.$$
Then  $$g^{(n)} = g_{n-1} \cdot \ldots \cdot g_1 \cdot g_0 = \Bn{n_1} (x)\,,$$
$$g_i \cdot g_{i-1} = \Bn{2 n_0} (\transl^{(i-1) n_0} \, x) $$ for $1 \le i \le n-2$, while 
$$g_{n-1} \cdot g_{n-2} = \Bn{m+n_0} (\transl^{(n-2) n_0} \, x)\,.$$

Since $x \notin \B_{n_0}^{ap} (B)$, we are in the setting described in Remark~\ref{bad-set-summary-o}, hence the AP (Proposition~\ref{AP-practical})  applies and we have:
\begin{align*}
d ( \mostexp ( \Bn{n_1} (x) ), \, \mostexp ( \Bn{n_0} (x) ) )  = d ( \mostexp (g^{(n)}), \, \mostexp (g_0) ) \less \frac{\varkappa_{n_0}}{\varepsilon_{n_0}}\,. \qquad    
\end{align*}
\end{proof}

\begin{lemma} \label{lemma2-most-dir}
For all $n_0 \ge \nzerobar$, \ $m \ge n_0^{2+}$, if $x \notin \B_{n_0}^{\flat} (B)$ then
\begin{equation} \label{eq-l1-most-dir}
d  ( \mostexp ( \Bn{m} (x) ), \, \mostexp ( \Bn{n_0} (x) ) ) \less  \frac{\varkappa_{n_0}}{\varepsilon_{n_0}}\,.
\end{equation}
\end{lemma}

\begin{proof}
Fix $0 < c \ll 1$. Let $\psi (t) := t^2$. 

Define inductively the following intervals of scales $\scale_0 := [ n_0^{1+c}, \, n_0^{3+c} ] \subset [ 2 n_0, \, \mes_{n_0}^{-1/2} ]$,  
$\scale_1 := \psi (\scale_0) =  [ n_0^{2+2c}, \, n_0^{6+2c} ]$ and for all $k \ge 0$, $\scale_{k+1} := \psi (\scale_k)$. 

These intervals overlap, so they cover up all integers $ \ge  n_0^{1+c}$.

Let $m \ge  n_0^{2+2c}$. Then there is $k \ge 0$ such that $m = m_{k+1} \in \scale_{k+1}$, and so $m_{k+1} \asymp m_k^2$ for some $m_k \in \scale_k$.  In fact, there is a backward "orbit'' of integers $m_0 \in \scale_0, m_1 \in \scale_1, \ldots,  m_k \in \scale_k$ such that $m_{j+1} \asymp m_j^2$.

For any $0 \le j \le k$, since $m_j \ge m_0 \ge n_0$, and since $x \notin \B_{n_0}^{\flat} (B)$, by \eqref{bad-set-apn-o} we  have $x \notin \B_{m_j}^{ap} (B)$. Moreover, since $m_{j+1} \asymp m_j^2$, we have $2 m_j \le m_{j+1} \le \mes_{m_j}^{-1/2}$, hence Lemma~\ref{lemma1-most-dir} is applicable to the scales  $m_j, m_{j+1}$ and we get
\begin{equation} \label{10-most-dir}
d  ( \mostexp ( \Bn{m_{j+1}} (x) ), \, \mostexp ( \Bn{m_{j}} (x) ) ) <  \frac{\varkappa_{m_j}}{\varepsilon_{m_{j}}}\,.
\end{equation}

Moreover, since $\mes_{n_0}^{-1/2} \ge m_0 \ge n_0^{1+c} \ge 2 n_0$, and since $x \notin \B_{n_0}^{\flat} (B)$, hence $x \notin \B_{n_0}^{ap} (B)$, Lemma~\ref{lemma1-most-dir} is applicable also at scales $m_0, n_0$, and we have
\begin{equation} \label{11-most-dir}
d ( \mostexp ( \Bn{m_{0}} (x) ), \, \mostexp ( \Bn{n_{0}} (x) ) ) <  \frac{\varkappa_{n_0}}{\varepsilon_{n_0}}\,.
\end{equation}

Using \eqref{11-most-dir}, \eqref{10-most-dir} and the triangle inequality we get
\begin{align*}
& d ( \mostexp ( \Bn{m} (x) ), \, \mostexp ( \Bn{n_0} (x) ) )  \\
&  \qquad  \le d ( \mostexp ( \Bn{m_{0}} (x) ), \, \mostexp ( \Bn{n_{0}} (x) ) ) \\
& \qquad + \sum_{j=0}^k \, d ( \mostexp ( \Bn{m_{j+1}} (x) ), \, \mostexp ( \Bn{m_{j}} (x) ) ) \\
& \qquad \le  \frac{\varkappa_{n_0}}{\varepsilon_{n_0}} + \sum_{j=0}^k \,  \frac{\varkappa_{m_j}}{\varepsilon_{m_{j}}} \less \frac{\varkappa_{n_0}}{\varepsilon_{n_0}}\,.
\end{align*}
The last inequality holds because by  \eqref{kappa over ep} we have $\frac{\varkappa_n}{\varepsilon_n} = e^{- 3 \, n \, \gapc(A) / 10}$, hence the series above converges rapidly, and so its  sum  is comparable with the first term. 
\end{proof}

From Proposition~\ref{med:aec} and by the proximity of the cocycle $B$ to $A$, we already know that its most expanding direction of $\medir{\infty} (B) (x)$ is well defined as the $\mu$-a.e. limit as $n\to\infty$ of the finite scale most expanding direction $\medir{n} (B) (x)$. We prove a quantitative version of this convergence.

\begin{proposition}[speed of convergence] \label{prop-speed-conv-most-dir}
For all $n \ge \nzerobar$, if $x \notin \B_n^\flat (B)$, hence for $x$ outside a set of measure $ < \mes_n^{1/2}$, we have
\begin{equation} \label{speed-conv-p-most-dir}
d  ( \mostexp ( \Bn{n} (x) ), \, \medir{\infty} (B) (x) ) < \frac{\varkappa_n}{\varepsilon_n}    = e^{- 3 \, n \gapc(A) / 10} <  \mes_n\,.
\end{equation}
Moreover,
\begin{equation} \label{speed-conv-most-dir}
\dist ( \medir{n} (B), \, \medir{\infty} (B)) < \mes_n^{1/2}.
\end{equation}
\end{proposition}

\begin{proof}
Estimate \eqref{speed-conv-p-most-dir} follows directly by taking the limit in lemma~\ref{lemma2-most-dir} with $n_0 = n$ and $m \to \infty$.
The second estimate follows by integration in $x$.
\end{proof}

\begin{proposition}[finite scale continuity] \label{finite-scale-cont-most-dir}
Given $C_2 > 0$, there is a constant $C_1 = C_1 (A, C_2) < \infty$, such that for any $B_1, B_2 \in \cocycles_m$ with $\dist (B_i, A) < \delta$, where $i = 1, 2$, if $n \ge \nzerobar$ and if $\dist (B_1, B_2) < e^{- C_1 \, n}$, then for  $x$ outside a set of measure $ < \mes_n$ we have
\begin{equation}\label{eq:-finite-scale-Oseledets-p}
\dist (\mostexp ( \Bn{n}_1 (x) ), \, \mostexp ( \Bn{n}_2 (x) ) ) < e^{- C_2 n } < \mes_n\,.
\end{equation}
Moreover,
\begin{equation}\label{eq:-finite-scale-Oseledets}
\dist ( \medir{n} (B_1), \medir{n} (B_2) ) < \mes_n\,.
\end{equation}
\end{proposition}

\begin{proof}
To prove \eqref{eq:-finite-scale-Oseledets-p} we will use the Lipschitz continuity of the most expanding singular direction in Proposition 3.18 from~\cite{LEbook-chap2}.
 
 Put $g_i = g_i (x) := \Bn{n}_i (x)$, \ $i = 1, 2$. 
 
Let $x \notin \B_n^{g} (B_1) \cup \B_n^{g} (B_2)$, which is a set of measure $ \less \mes_n$.  We show that the assumptions of Proposition 3.18 in~\cite{LEbook-chap2} hold for all such $x$.

Firstly note that by \eqref{1-most-dir} we have
$  \rgap (g_i) = \rgap (\Bn{n}_i (x)) >  \frac{1}{\kappa_n} > 1,$
so in particular $\mostexp ( \Bn{n}_i (x) )$ are well defined.

Moreover, the fiber-LDT estimate applies to $B_i$ and we have
$$\frac{1}{n} \, \log \norm{\Bn{n}_i (x) } > \Lan{n}{1} (B_i) - \ep_n > \Lan{n}{1} (A) - \frac{\gapc (A)}{20} - \frac{\gapc (A)}{100} > - C_0,$$
where we used  \eqref{2-most-dir} in the estimate above, and $C_0 = C_0 (A) < \infty$.

Then
$$ \norm{g_i} = \norm{\Bn{n}_i (x) }  > e^{- C_0 \, n}.$$

Since for $\mu$ almost every $x$, $\norm{B_i (x)} < C(A) < \infty$, by possibly increasing $C_0$, we may also assume that 
$$\norm{g_i} = \norm{\Bn{n}_i (x) } < e^{C_0 \, n}.$$

Moreover, assuming $\dist (B_1, B_2) < e^{- C_1 \, n}$, with $C_1$ to be chosen later,
\begin{align*}
\norm{ g_1 - g_2 } = \norm{ \Bn{n}_1 (x) - \Bn{n}_2 (x)  } \le n \, C_0^{n-1} \, \dist (B_1, B_2) < e^{- (C_1 - 2 \log C_0) \, n}.
\end{align*}

Then 
$$ \drel(g_1, g_2):=\frac{ \norm{g_1-g_2} }{ \max\{ \norm{g_1}, \norm{g_2} \}} \le \frac{e^{- (C_1 - 2 \log C_0) \, n}}{e^{-C_0 n}} < e^{- C_2 \, n} \ll 1\,,$$
provided we choose $C_1 > 2 \log C_0 + C_0 + C_2$.

Proposition 3.18 in~\cite{LEbook-chap2} applies, and we conclude:
$$d ( \mostexp (g_1), \mostexp (g_2) ) \leq \frac{16}{1-\kappa_n^2}\, \drel(g_1,g_2)  \less e^{- C_2 n}. $$

This proves \eqref{eq:-finite-scale-Oseledets-p}, while \eqref{eq:-finite-scale-Oseledets} follows by integration in $x$.

\end{proof}

We are now ready to formulate and to prove the continuity of the most expanding direction.

\begin{theorem}\label{thm:most-dir}
Let $A \in \cocycles_m$ with $L_1 (A) > L_2 (A)$. There are $\delta > 0$, $\mesf \in \mesfs$, $c  > 0$, $\alpha > 0$, all depending only on $A$, such that for any cocycles $B_1, B_2 \in \cocycles_m$ with $\dist (B_i, A) < \delta$, where $i = 1, 2$,  we have:
\begin{align*} \label{dir cont pointwise}
 \mu  \ \{  x\in X \colon d ( \medir{\infty} (B_1) (x), \  \medir{\infty} (B_2) (x) ) > \dist (B_1, B_2)^\alpha \}  
 < \omega_{\mesf} (\dist (B_1, B_2))
\end{align*}
where $\omega_{\mesf} (h) := [ \mesf \, (c \, \log  ( 1/ h ) ) ]^{1/2}$ is a modulus of continuity function, and clearly $\omega_{\mesf} (h) \to 0$ as $h \to 0$.

Moreover,
\begin{equation} \label{Oseledets:modulus-cont} 
\dist (\medir{\infty}  (B_1), \medir{\infty}  (B_2) ) <  \omega_{\mesf} (\dist (B_1, B_2))\,.
\end{equation}
\end{theorem}

\begin{proof} Fix any $C_2 > 0$ and let  $C_1$ be the constant in Proposition~\ref{finite-scale-cont-most-dir}.

Put $\dist (B_1, B_2) =: h$ and choose $n \in \N$ such that $h \asymp e^{- C_1 \, n}$.
Since $h \le  2 \delta$ and $n \asymp 1/C_1 \, \log 1/h$, by taking $\delta$ small enough we may assume that $n \ge \nzerobar$.

Apply Proposition~\ref{finite-scale-cont-most-dir} to get that for $x$ outside a set of measure $< \mes_n$,
\begin{equation} \label{eq100-most-dir}
\dist (\mostexp ( \Bn{n}_1 (x) ), \, \mostexp ( \Bn{n}_2 (x) ) ) < e^{- C_2 n } \,.
\end{equation}

Now apply Proposition~\ref{prop-speed-conv-most-dir} to $B= B_i$, $i=1, 2$, to get that for $x$ outside a set of measure $< \mes_n^{1/2}$,
\begin{equation} \label{eq101-most-dir}
d  ( \mostexp ( \Bn{n}_i (x) ), \, \medir{\infty} (B_i) (x) ) < e^{- 3 \, n \gapc(A) / 10}. 
\end{equation}

Combine \eqref{eq100-most-dir} and \eqref{eq101-most-dir} to conclude that for $x$ outside a set of measure $\less  \mes_n^{1/2}$ and for  $c_0 < \min \{ C_2, 3 \gapc(A) /10 \}$, we have
$$
d ( \medir{\infty} (B_1) (x), \  \medir{\infty} (B_2) (x) ) < e^{- c_0 \, n} = h^\alpha\,,
$$
where $\alpha = \frac{c_0}{C_1}$. This proves the pointwise estimate. 

To prove \eqref{Oseledets:modulus-cont}, simply integrate in $x$ and take into account the fact that since the large deviation measure function $\mesf$ decays at most exponentially, the corresponding modulus of continuity function $\omega_{\mesf}$ will decay at most like a power of $h$, so we may assume $h^\alpha < \omega_{\mesf} (h)$ as $h \to 0$.

\end{proof}

\subsection{Spaces of measurable  filtrations and decompositions}
\label{smf}
We introduce a space of measurable filtrations, i.e. a space of functions from the phase space to the set of all flags.  Thus the Oseledets filtration of a linear cocycle is an element of this space. We endow the space of measurable filtrations with a natural topology. Similarly, we define a space of measurable decompositions.

\medskip

We start with an example that will  motivate the formalism below. Let $A$ be a linear cocycle with {\em exact} gap pattern say $\tau = (2, 3)$, that is,
$$L_1 (A) = L_2 (A) > L_3 (A) > L_4 (A) = \ldots = L_m (A).$$

The Oseledets filtration of $A$ is a $\tau^\perp = (m-3, m-2)$-flag  
$$\{0\}= F_{4} (A) (x) \subsetneq F_3 (A) (x)
 \subsetneq F_{2} (A) (x) \subsetneq F_1 (A) (x) = \R^m\,,$$
 for $\mu$-a.e. $x \in X$, thus defining a measurable function $F (A) \colon X \to \FF_{\tau^\perp} (\R^m)$.
 
The growth rate of the iterates of $A$ along vectors in $F_{3} (A) (x)$ is $L_4 (A)$ or less, the growth rate along vectors in $F_{2} (A) (x)$ is $L_3 (A)$ or less and the growth rate along vectors in $F_{1} (A) (x)$ is $L_1 (A)$ or less.
 
 By the continuity of the Lyapunov exponents of a linear cocycle (which holds under the assumptions in this section), if $B$ is a small perturbation of $A$, then 
 $$L_1 (B) \ge L_2 (B) > L_3 (B) > L_4 (B) \ge \ldots \ge L_m (B),$$
 meaning that $B$ will still have a $\tau = (2, 3)$ gap pattern. 
 However, this might not be its {\em exact} gap pattern, as we could have $L_1 (B) > L_2 (B)$, leading to a {\em finer} gap pattern, say $\tau' = (1, 2, 3)$. 
  If $\tau'$ were the exact gap pattern of $B$, then its Oseledets filtration would be a $\tau'^\perp = (m-1, m-2, m-3)$-flag 
 $$\{0\}= F_{5} (B) (x) \subsetneq F_4 (B) (x) \subsetneq F_3 (B) (x)
 \subsetneq F_{2} (B) (x) \subsetneq F_1 (B) (x) = \R^m\,,$$ 
 for $\mu$-a.e. $x \in X$, thus defining a measurable function $F (B) \colon X \to \FF_{\tau'^\perp} (\R^m)$.
 
 The subspaces $F_4 (B) (x), F_3 (B) (x), F_2 (B) (x)$ and  $F_1 (B) (x)$ correspond to the Lyapunov exponents $L_4 (B), L_3 (B), L_2 (B)$ and $L_1 (B)$ respectively.
 
 In order to compare the Oseledets filtration of $B$ with that of $A$, we would need to ``forget" the extra subspace $F_2 (B) (x)$ corresponding to the Lyapunov exponent $L_2 (B)$, which appears precisely because the gap pattern $\tau'$ of $B$ is finer than that of $A$. In other words, we consider the projection  $F^{\tau} (B)$  of the Oseledets filtration $ F(B) $ to the space of coarser $\tau^\perp = (m-3, m-2)$-flags valued filtrations  
 $$\{0\}= F_{5} (B) (x) \subsetneq F_4 (B) (x) \subsetneq F_3 (B) (x)
 \subsetneq F_1 (B) (x) = \R^m\,.$$ 
 
Now  $F (A) (x)$ and  $F^{\tau} (B) (x)$ are both $\tau^\perp$-flags, and we may define a distance between them component-wise  (as points in the same Grassmann manifold). The distance between the measurable filtrations  $F (A)$ and  $F^{\tau} (B)$  as functions on $X$ will be the space average of the pointwise distances. 

\smallskip

Furthermore, the Oseledets decomposition $\dec (A)$ of the cocycle $A$ with exact $\tau = (2, 3)$ gap pattern, consists of a $2$-dimensional subspace $E_1 (A) (x)$ corresponding to $L_1 (A) = L_2 (A)$, a one dimensional subspace $E_2 (A) (x)$ corresponding to $L_3 (A)$, and an $m-3$-dimensional subspace $E_3 (A) (x)$ corresponding to the remaining (and equal) Lyapunov exponents. 

If a small perturbation $B$ of $A$ has (as above) the finer $\tau' = (1, 2, 3)$ gap pattern, then its Oseledets decomposition will consist of subspaces $E_1 (B) (x)$ (one dimensional, corresponding to $L_1 (B)$), $E_2 (B) (x)$ (one dimensional, corresponding to $L_2 (B)$), $E_3 (B) (x)$ (one dimensional, corresponding to $L_3 (B)$) and the subspace $E_4 (B) (x)$ ($m-3$ dimensional, corresponding to the remaining Lyapunov exponents).

In order to compare the Oseledets decompositions of $A$ and $B$, we would have to ``patch up" the first two Oseledets subspaces for $B$. In other words, we will consider the natural restriction $\dec^\tau (B)$ of the Oseledets decomposition $\dec (B)$, consisting of the subspaces $E_1 (B) \oplus E_2 (B)$, $E_3 (B)$, $E_4 (B)$.  

\medskip

We  make the obvious observation that for {\em two dimensional} (i.e. $\Mat (2, \R)$-valued) cocycles,  or for cocycles of any dimension with {\em simple} Lyapunov spectrum, these projection/restriction of the filtration /decomposition are not needed. 

\medskip

Let us now formally define the space of measurable filtrations.
 
Given two signatures
$\tau=(\tau_1,\ldots, \tau_k)$  and
$\tau'=(\tau_1',\ldots, \tau_{k'}')$, we say that $\tau$ refines $\tau'$,
and write $\tau\geq \tau'$, if 
$\{\tau_1,\ldots, \tau_k\}\supseteq \{\tau_1',\ldots, \tau_{k'}'\}$.

Given $\tau\geq \tau'$, there is a natural   projection $\rho_{\tau,\tau'}:\FF_\tau(\R^m)\to \FF_{\tau'}(\R^m)$,
defined by 
$$ \rho_{\tau,\tau'}(\filt) = 
\rho_{\tau,\tau'}(F_1,\ldots, F_k):= (F_{i_1},\ldots, F_{i_{k'}}) \;,$$
where $\tau_j'=\tau_{i_j}$ for $j=1,\ldots, k'$.

With respect to the following normalized distance  on the flag manifold $\FF_\tau(\R^m)$ (see (1.11) in~\cite{LEbook-chap2}), 
\begin{equation*} 
 d_\tau(\filt,\filt') =\max_{1\leq j\leq k} d(F_j,F_j')  
\end{equation*}
these projections are Lipschitz, with Lipschitz constant $1$.

Let $(X,\mathcal{F},\mu)$ be a probability space and $\transl:X\to X$ be an ergodic measure preserving transformation. 


We call {\em measurable filtration of $\R^m$} any mod $0$ equivalence class of an $\mathscr{F}$-measurable function  $\filt:X\to \FF(\R^m)$.
Two functions $\filt,\filt':X\to \FF(\R^m)$ are said to be equivalent mod $0$ when
$\filt(x)=\filt'(x)$ for $\mu$-a.e. $x\in X$. From now on we will identify each mod $0$ equivalence class with any of its
representative measurable functions.

Given any measurable filtration $\filt$ of $\R^m$, let $\tau (\filt) (x)$ denote the signature of the flag $\filt (x)$. We say that 
$\filt$ has a
$\transl$-invariant signature if $\tau (\filt) (x) = \tau (\filt) (T x)$
for $\mu$-a.e. $x\in X$. If this is the case, then by the ergodicity of $(\transl,\mu)$ the function $\tau (\filt) (x)$ is constant $\mu$-a.e.

Define $\Filt(X,\R^m)$ to be the space of mod $0$ equivalence classes of measurable filtrations
 with a $\transl$-invariant signature, which is a constant that we denote by $\tau(\filt)$. 
 
 We say that $\filt$ has a {\em $\tau$-pattern} when
$\tau(\filt)\geq \tau$.

Given a signature $\tau$, let us define
  $\Filatp{\tau}(X,\R^m)$ to be the subspace of  measurable filtrations in
$\Filt(X,\R^m)$ with a $\tau$-pattern.

By definition  $\Filatp{\tau}(X,\R^m)\subseteq \Filatp{\tau'}(X,\R^m)$, whenever $\tau\geq \tau'$.

Given $\filt\in \Filatp{\tau}(X,\R^m)$, the function
$$\filt^\tau(x):=\rho_{\tau(\filt),\tau}(\filt(x))$$
determines a measurable filtration with constant signature $\tau$, which will be referred to as the $\tau$-restriction of $\filt$. 

We endow $\Filatp{\tau}(X,\R^m)$  with the following distance  
\begin{equation}\label{delta:tau}
\dist_\tau(\filt, \filt'):= \int_X d_\tau\left( \filt^\tau(x),(\filt')^\tau(x)\right)\, d\mu(x)\;. 
\end{equation}

Finally, we endow the space $\Filt(X,\R^m)$ of all measurable filtrations of $\R^m$ with the topology
determined by the following neighborhood bases,
$$ \mathcal{V}_{\delta,\tau}(\filt):=\{\,
\filt'\in \Filatp{\tau}(X,\R^m)\,:\,
\dist_\tau(\filt,\filt')<\delta\, \} \;,$$
where $\delta>0$,  $\filt\in \Filt(X,\R^m)$   and $\tau=\tau(\filt)$.

We note that this topology is not metrizable.

\begin{proposition}\label{topology-meas-filt}
Let $\mathcal{C}$ be a topological space.
A map $\filt :\mathcal{C}  \to \Filt(X,\R^m)$  is continuous  w.r.t. this topology if and only if
for all $A\in\mathcal{C}$ such that $\filt(A)$ has a $\tau$-pattern, there exists a neighborhood $\mathscr{U}\subset \mathcal{C}$
of $A$ such that 
$\filt(\mathscr{U})\subseteq \Filatp{\tau}(X,\R^m)$ and the $\tau$-restricted function
$\filt^\tau\vert_{\mathscr{U}}:\mathscr{U}\to  \Filatp{\tau}(X,\R^m)$, $B\mapsto \filt^\tau(B)$, is continuous
w.r.t. the distance $\dist_\tau$ defined above.
\end{proposition}

\begin{proof}
Assume first that $\filt :\mathcal{C}  \to \Filt(X,\R^m)$  is continuous and take $A\in \mathcal{C}$ such that 
$\filt(A)\in \Filt(X,\R^m)$. Consider the neighborhood
$\mathcal{V}_{\delta,\tau}(\filt(A))$ of $\filt(A)$ where $\delta>0$.
By continuity of $\filt$, there exists a neighborhood $\mathscr{U}\subset \mathcal{C}$ of $A$ such that   $\filt(\mathscr{U})\subset
\mathcal{V}_{\delta,\tau}(\filt(A))\subset \Filatp{\tau}(X,\R^m)$.
By definition of the topology in $\Filt(X,\R^m)$, the set $\Filatp{\tau}(X,\R^m)$ is open in $\Filt(X,\R^m)$, and the projection 
$\rho_\tau:\Filatp{\tau}(X,\R^m)\to \Filatp{\tau}(X,\R^m)$,
$\rho_\tau(\filt)=\filt^\tau$, is continuous.
The restriction $\filt^\tau\vert_{\mathscr{U}}:\mathscr{U}\to  \Filatp{\tau}(X,\R^m)$, $B\mapsto \filt^\tau(B)$, is continuous because it coincides with the composition $\rho_\tau\circ \filt$.

The converse statement is a direct consequence of the definition.
\end{proof}

\medskip

Recall that  a $\tau$-decomposition is a family of linear subspaces $E_{\cdot}=\{E_i\}_{1\leq i\leq k+1}$ in $\Gr(\R^m)$ such that $\R^m=\oplus_{i=1}^{k+1} E_i$ and
$\dim E_i=\tau_i-\tau_{i-1}$ for all
$1\le i\leq k+1$. In~\cite{LEbook-chap2}, we denoted 
by $\Decompsp_\tau(V)$ the space of all  $\tau$-decompositions of $\R^m$.

Given $\tau\geq \tau'$, there is a natural   projection $\rho_{\tau,\tau'}:\Decompsp_\tau(\R^m)\to \Decompsp_{\tau'}(\R^m)$,
defined by 
$$ \rho_{\tau,\tau'}(E_{\cdot}) = 
\rho_{\tau,\tau'}(E_1,\ldots, E_{k+1}):= (E_{1}',\ldots, E_{k'+1}') \;,$$
where $E_j'=\oplus_{i_j\leq l <i_{j+1}} E_l$ and $\tau_{i_j}=\tau_j'$ for $j=1,\ldots, k'$.

On the  space of decompositions  $\Decompsp_\tau(\R^m)$ 
we consider the distance  
(see Definition 3.3 in~\cite{LEbook-chap2}), 
\begin{equation*} 
 d_\tau(E_{\cdot},E_{\cdot}') =\max_{1\leq j\leq k+1} d_{\tau_j-\tau_{j+1}}(E_j,E_j')  \;.
\end{equation*}
By Proposition 3.6 in~\cite{LEbook-chap2}, 
the projections $\rho_{\tau,\tau'}$ are locally Lipschitz.

An equivalence class mod $0$ of an $\mathscr{F}$-measurable function  $E_{\cdot}:X\to \Decompsp(\R^m):=\cup_{\tau} \Decompsp_\tau(\R^m)$ will be called a {\em measurable decomposition of $\R^m$}. 
Two measurable decompositions $E_{\cdot},E_{\cdot}':X\to \Decompsp(\R^m)$ are equivalent mod $0$ when
$E_{\cdot}(x)=E_{\cdot}'(x)$ for $\mu$-a.e. $x\in X$. As before, we will identify each mod $0$ equivalence class with any of its
representative measurable functions.

Given any measurable decomposition $E_{\cdot}$ of $\R^m$, 
its signature at a point $x\in X$ is the sequence of dimensions
$\tau=(\tau_1,\ldots, \tau_k)$, where $\tau_j= \dim \left(\oplus_{l\leq j} E_j(x)\right)$ for all $j=1,\ldots, k$.
We denote it  by $\tau (E_{\cdot}) (x)$. We say that 
$E_{\cdot}$ has a
$\transl$-invariant signature if $\tau (E_{\cdot}) (x) = \tau (E_{\cdot}) (T x)$
for $\mu$-a.e. $x\in X$. In this case,  by the ergodicity of $(\transl,\mu)$ the function $\tau (E_{\cdot}) (x)$ is constant $\mu$-a.e.

Define $\Dec(X,\R^m)$ to be the space of mod $0$ equivalence classes of measurable decompositions
 with a $\transl$-invariant signature, that we denote by $\tau(E_{\cdot})$. 
 
 We say that $E_{\cdot}$ has a {\em $\tau$-pattern} when
$\tau(E_{\cdot})\geq \tau$.

Given a signature $\tau$,  define
  $\Decatp{\tau}(X,\R^m)$ to be the subspace of  measurable decompositions in
$\Dec(X,\R^m)$ with a $\tau$-pattern.

By definition  $\Decatp{\tau}(X,\R^m)\subseteq \Decatp{\tau'}(X,\R^m)$, whenever $\tau\geq \tau'$.

Given $E_{\cdot}\in \Decatp{\tau}(X,\R^m)$, the function
$$E_{\cdot}^\tau(x):=\rho_{\tau(E_{\cdot}),\tau}(E_{\cdot}(x))$$
determines a measurable decomposition with constant signature $\tau$,  referred to as the $\tau$-restriction of $E_{\cdot}$. 

We endow $\Decatp{\tau}(X,\R^m)$  with the following distance  
\begin{equation}\label{delta:tau}
\dist_\tau(E_{\cdot}, E_{\cdot}'):= \int_X d_\tau\left( E_{\cdot}^\tau(x),(E_{\cdot}')^\tau(x)\right)\, d\mu(x)\;. 
\end{equation}

Finally, we endow the space $\Dec(X,\R^m)$ of all measurable decompositions of $\R^m$ with the topology
determined by the following neighborhood bases,
$$ \mathcal{V}_{\delta,\tau}(E_{\cdot}):=\{\,
E_{\cdot}'\in \Decatp{\tau}(X,\R^m)\,:\,
\dist_\tau(E_{\cdot},E_{\cdot}')<\delta\, \} \;,$$
where $\delta>0$,  $E_{\cdot} \in \Dec(X,\R^m)$   and $\tau=\tau(E_{\cdot})$.

Again, this topology is not metrizable, and a similar characterization
 of the continuity of a map $E_{\cdot} :\mathcal{C}  \to \Dec(X,\R^m)$ holds.

\begin{proposition}\label{topology-meas-decomp}
Let $\mathcal{C}$ be a topological space.
A map $E_{\cdot} :\mathcal{C}  \to \Dec(X,\R^m)$  is continuous  w.r.t. this topology if and only if
for all $A\in\mathcal{C}$ such that $E_{\cdot}(A)$ has a $\tau$-pattern, there exists a neighborhood $\mathscr{U}\subset \mathcal{C}$
of $A$ such that 
$E_{\cdot}(\mathscr{U})\subseteq \Decatp{\tau}(X,\R^m)$ and the $\tau$-restricted function
$E_{\cdot}^\tau\vert_{\mathscr{U}}:\mathscr{U}\to  \Decatp{\tau}(X,\R^m)$, $B\mapsto E_{\cdot}^\tau(B)$, is continuous
w.r.t. the distance $\dist_\tau$ defined above.
\end{proposition}

\subsection{Continuity of the Oseledets filtration}
\label{cof}
We denote by $\filt (A)$ the Oseledets filtration of a cocycle $A \in \cocycles_m$. If $A$ has a $\tau$ gap pattern, by the continuity of the Lyapunov exponents, any nearby cocycle $B$ has the same or a finer gap pattern $\tau' \ge \tau$. Let $\filt^{\tau} (B)$ denote the projection of the Oseledets filtration of $B$ to the space $\Filatp{\tau}(X,\R^m)$ of measurable filtrations with a $\tau$-pattern. We are now ready to phrase and to prove the continuity of the Oseledets filtration.

\begin{theorem}\label{cont-oseledets-filt-o}
Let $A \in \cocycles_m$ be a cocycle with a $\tau$ gap pattern. Then locally near $A$, the map $$\cocycles_m \ni B \mapsto \filt^{\tau} (B) \in \Filatp{\tau}(X,\R^m)$$ is continuous with a modulus of continuity $\omega (h) :=  [ \mesf \ (c \, \log \frac{1}{h}) ]^{1/2}$ for some constant $c = c (A) > 0$ and for some deviation measure function  $\mesf = \mesf (A)$ from the corresponding set of LDT parameters.

In fact, a stronger pointwise estimate holds:
$$\mu  \ \{  x \in X \colon d ( \filt^\tau (B_1) (x), \  \filt^\tau (B_2) (x) ) > \dist (B_1, B_2)^\alpha \}  
 < \omega (\dist (B_1, B_2))\,,$$
for any $B_1, B_2 \in \cocycles_m$ in a neighborhood of $A$, and for $\alpha = \alpha (A) >0$.   

\smallskip

Moreover, the map $\cocycles_m \ni A \mapsto \filt (A) \in \Filt(X,\R^m)$ is continuous everywhere.
\end{theorem} 

\begin{proof}
Since $A$ has a $\tau = (\tau_1, \ldots, \tau_{k})$ gap pattern, $L_{\tau_{j}} (A) > L_{\tau_{j-1}} (A)$  for all indices $j$, so $L_1 (\wedge_{\tau_j} A) > L_2 ( \wedge_{\tau_j} A)$. We may then apply the continuity of the most expanding direction in Theorem~\ref{thm:most-dir} to $\wedge_{\tau_j} A$ and obtain that
$$\cocycles_m \ni B \mapsto \wedge_{\tau_j} B \mapsto \medir{\infty} (\wedge_{\tau_j} B)  \in L^1 (X, \Pp (\wedge_{\tau_j} \R^m))$$
is continuous at $A$, with a modulus of continuity of the form $\omega (h) =  [ \mesf \ (c \, \log \frac{1}{h}) ]^{1/2}$.

A similar pointwise estimate holds as well.

\medskip

The Oseledets filtration of $A$ was obtained in the proof of the Oseledets Theorem~\ref{Oseledets non invertible} as $F (A) (x) = \bigl[ \medir{\infty}_\tau (A) (x) \bigr]^\perp$, where
$$\medir{\infty}_{\tau} (A) (x) = \left( \medir{\infty}_{\tau_1} (A) (x), \ldots, \medir{\infty}_{\tau_k} (A) (x)  \right)\,,$$
and
$$\medir{\infty}_{\tau_j} (A) (x) = \Psi^{-1} (\medir{\infty} (\wedge_{\tau_j} A) (x) )\,.$$

Moreover, since for any nearby cocycle $B$ we clearly have
$$ \filt^{\tau} (B) (x) =  \left( \Psi^{-1} ( \medir{\infty} (\wedge_{\tau_1} B) (x) ), \ldots,   \Psi^{-1} (\medir{\infty} (\wedge_{\tau_k} B) (x) ) \right)^\perp\,,$$ 
the first two assertions follow from the continuity of the most expanding direction and the fact that the Pl\"ucker embedding $\Psi$ and the orthogonal complement $\perp$ are  isometries.
The third assertion is an immediate consequence of Proposition~\ref{topology-meas-filt}.
\end{proof}

\subsection{Continuity of the Oseledets decomposition}
\label{cod}
We denote by $\dec (A)$ the Oseledets decomposition of a cocycle $A \in \cocycles_m$. Assume that $A$ has a $\tau = (\tau_1, \ldots, \tau_k)$ gap pattern. By the construction in the proof of Theorem~\ref{thm:oseledets2}, we have
$$\dec (A) (x) = \medir{\infty}_\tau (A^\ast) (x) \sqcap \medir{\infty}_\tau (A) (x)^\perp\,.$$ 

By the continuity of the Lyapunov exponents, any nearby cocycle $B$ has the same or a finer gap pattern $\tau' \ge \tau$. Let $\dec^\tau (B) (x)$ denote the $\tau$-restriction of $\dec (B) (x)$ to the space of decompositions with signature $\tau$. Clearly we have
$$\dec^\tau (B) (x) = \medir{\infty}_\tau (B^\ast) (x) \sqcap \medir{\infty}_\tau (B) (x)^\perp\,.$$

We may immediately conclude from Subsection~\ref{cof}, or directly from he continuity of the most expanding direction derived in Subsection~\ref{direction} that the maps
$$B \mapsto \medir{\infty}_\tau (B^\ast) \quad \text{and} \quad B \mapsto  \medir{\infty}_\tau (B)^\perp$$
are continuous in a neighborhood of $A$, with an appropriate modulus of continuity.

However, this does not automatically guarantee the continuity of the intersection. Indeed, by Proposition 3.16 in~\cite{LEbook-chap2},  the intersection map
$\sqcap:\FF_\tau(V)\times \FF_{\tau^\perp}(V)\to \Decompsp_\tau(V)$ is locally Lipschitz, but with a Lipschitz constant that depends on the transversality measurement of the subspaces, which may blow up for some phases $x$. 

That is why we need to control these transversality measurements at {\em finite scale} first. We will employ a similar scheme as in the establishing of the  continuity of the most expanding direction in Subsection~\ref{direction}.

Recall from Subsection~\ref{met} the $n$-th scale partial functions
$\mostexp^{(n)}_\tau(B)$ on $X$ taking values on $\FF_\tau(\R^m)$,
$$ \mostexp^{(n)}_\tau(B)(x):=
\left\{ \begin{array}{lll}
\mostexp_\tau(B^{(n)}(x)) & & \text{ if } \ \rgap_\tau(B^{(n)}(x))>1 \\
\text{undefined} & &  \text{otherwise},
\end{array} \right.$$
where
\begin{align*}
\mostexp_\tau (\Bn{n} (x)) & = \left( \mostexp_{\tau_1}  (\Bn{n} (x)), \ldots, \mostexp_{\tau_k}  (\Bn{n} (x)) \right)\\
& = \left( \Psi^{-1} (\mostexp (\wedge_{\tau_1} \Bn{n} (x))), \ldots,  \Psi^{-1} (\mostexp (\wedge_{\tau_k} \Bn{n} (x))) \right) .
\end{align*}

Consider the exceptional sets defined in Subsection~\ref{direction} for each dimension $\tau_j$, that is, define
$$\B^\sharp_n (B) := \bigcup_{1\le j \le k} \B^\flat_n (\wedge_{\tau_j} B) .$$

Since $A$ has a $\tau = (\tau_1, \ldots, \tau_k)$ gap pattern, the estimates on the most expanding direction, namely Remark~\ref{bad-set-summary-o}, Proposition~\ref{prop-speed-conv-most-dir} and Proposition~\ref{finite-scale-cont-most-dir} are applicable to $\wedge_{\tau_j} B$, $1 \le j \le k$. We summarize the relevant results in the following remark.

\begin{remark}\label{summary-estimates-dec}\normalfont
There are parameters $\delta, \nzerobar$ and $\mesf$, depending only on $A$, such that the following hold for all cocycles $B$ with $\dist (B, A) < \delta$ and for all scales $n \ge \nzerobar$. 

\begin{enumerate}
\item  $\mostexp_\tau(B^{(n)}(x))$ is well defined for all phases $x \notin \B^\sharp_n (B)$. Moreover, for all such $x$  we have
\begin{align*}
\rgap_\tau(B^{(n)}(x)) &= \min_{1 \le j \le k} \rgap (\wedge_{\tau_j} \Bn{n} (x))  > \frac{1}{\varkappa_n}\,, \\
\alpha_\tau (\Bn{n} (T^{-n} x), \Bn{n} (x) ) & \asymp \min_{1 \le j \le k}   \frac{ \norm{ \wedge_{\tau_j} \Bn{2 n} (x)}}{\norm{ \wedge_{\tau_j} \Bn{n} (\transl^{-n} x)} \cdot \norm{\wedge_{\tau_j} \Bn{n} (x)}}   > \varepsilon_n\,. 
\end{align*}

\item The sequence of partial functions $\mostexp^{(n)}_\tau(B)$  converges $\mu$-a.e, as $n \to \infty$,  to a function $\mostexp^{(\infty)}_\tau(B) \colon X \to \FF_\tau(\R^m)$. 

\item For all phases $x \notin \B^\sharp_n (B)$, we have the following rate of convergence:
\begin{equation}\label{speed conv cod}
d_\tau \left(  \mostexp^{(n)}_\tau(B)(x),  \mostexp^{(\infty)}_\tau(B)(x) \right) < \frac{\varkappa_n}{\varepsilon_n}\,.
\end{equation}

\item The partial functions $\mostexp^{(n)}_\tau(B)$ satisfy the following finite scale uniform continuity property. Given $C_2 > 0$, there is $C_1 = C_1 (A, C_2) < \infty$ such that for any cocycles $B_i \in \cocycles_m$, $\dist (B_i, A) < \delta$ for $i=1, 2$, if $\dist (B_1, B_2) < e^{ - C_1 \, n}$, then for $x$ outside a set of measure $ < \mes_n$ we have:
\begin{equation}\label{finite scale cod}
d_\tau \left(  \mostexp^{(n)}_\tau(B_1)(x),  \mostexp^{(n)}_\tau(B_2)(x) \right) < e^{- C_2 \, n}.
\end{equation}
\end{enumerate}
\end{remark}

\begin{proof} The statements in item 1 above follow from \eqref{1-most-dir} and \eqref{3-most-dir} applied to  $ \wedge_{\tau_j} B$ for $1 \le j \le k$. 

Each component of the flag $\mostexp_\tau (\Bn{n} (x))$ converges, for $\mu$-a.e. $x \in X$, by Proposition~\ref{med:aec} and the fact that $B$ has the $\tau$ gap pattern.

The rate of convergence in item 3 is a consequence of Proposition~\ref{prop-speed-conv-most-dir} applied in each component of the flag $\mostexp_\tau (\Bn{n} (x))$, that is, applied to the exterior powers $ \wedge_{\tau_j} B$ for $1 \le j \le k$. The same argument holds for item 4. 

\end{proof}

\begin{remark} \normalfont Since $A$ has the $\tau$ gap pattern, so does $A^\ast$. Therefore, by possibly doubling the size of the exceptional set, we may assume that the rate of convergence \eqref{speed conv cod} holds for both $B$ and $B^\ast$. The same applies to the finite scale continuity \eqref{finite scale cod}. 
\end{remark}

We define a finite scale decomposition which will be shown to converge to the ($\tau$ restricted) Oseledets decomposition.

Consider the partial function on $X$ taking values in $\Decompsp_\tau(\R^m)$ and defined by
$$ \dec^{(n)} (B)(x):=
\left\{ \begin{array}{lll}
 \mostexp^{(n)}_\tau(B^\ast)(x) \sqcap  \mostexp^{(n)}_\tau(B)(x)^\perp& & \text{ if } \ \rgap_\tau(B^{(n)}(x))>1 \\
\text{undefined} & &  \text{otherwise}.
\end{array} \right.$$

Clearly this map is well defined for all  $x \notin \B^\sharp_n (B)$.

We begin by establishing a lower bound on the transversality measurement for the flags defining this finite scale decomposition.

\begin{lemma} \label{lemma transv bound cod}
For all $x \notin \B^\sharp_n (B)$ and $n \ge \nzerobar$ we have
\begin{equation} \label{eq transv cod}
\theta_{\sqcap}  \left(  \mostexp^{(n)}_\tau(B^\ast)(x),  \mostexp^{(n)}_\tau(B)(x)^\perp \right) \ge \varepsilon_n\,.
\end{equation}
\end{lemma}
\begin{proof}
This lower bound follows easily from Proposition 3.17 in~\cite{LEbook-chap2} and the second inequality in item 1 of Remark~\ref{summary-estimates-dec}.
\begin{align*}
\theta_{\sqcap}  \left(  \mostexp^{(n)}_\tau(B^\ast)(x),  \mostexp^{(n)}_\tau(B)(x)^\perp \right) & = 
\theta_{\sqcap}  \left(  \mostexp_\tau({B^\ast}^{(n)} (x) ),  \mostexp_\tau(B^{(n)} (x))^\perp \right) \\
& = \theta_{\sqcap}  \left(  \mostexp_\tau({B^{(n)} (\transl^{- n} x) }^\ast),  \mostexp_\tau(B^{(n)} (x))^\perp \right) \\
& \ge \alpha_\tau (\Bn{n} (T^{-n} x), \Bn{n} (x) ) \ge \varepsilon_n\,. \quad 
\end{align*}
\end{proof}

Next we establish the convergence to $\dec^\tau (B)$ of the finite scale decomposition introduced above.

\begin{proposition}[speed of convergence] \label{prop speed conv cod}
For all $x \notin \B^\sharp_n (B)$ and $n \ge \nzerobar$ we have
\begin{equation}\label{eq speed conv cod}
d \left(  \dec^{(n)} (B)(x), \dec^\tau (B) (x) \right) < \frac{\varkappa_n}{\varepsilon_n^2}\,.
\end{equation}
\end{proposition}

\begin{proof}
Fix the phase $x$ and the scale $n$. For simplicity of notation let
\begin{align*}
F:=  \mostexp^{(\infty)}_\tau(B^\ast)(x) \in  \FF_\tau (\R^m), \quad & F' := \mostexp^{(\infty)}_\tau(B)(x)^\perp \in  \FF_{\tau^\perp} (\R^m), \\
F_0:=  \mostexp^{(n)}_\tau(B^\ast)(x) \in  \FF_\tau (\R^m), \quad & F_0^{'} := \mostexp^{(n)}_\tau(B)(x)^\perp \in  \FF_{\tau^\perp} (\R^m).
\end{align*}

With these notations we have $ \dec^\tau (B) (x) = F \sqcap F'$ and $ \dec^{(n)} (B)(x) = F_0 \sqcap F_0^{'}$.

By Proposition 3.16 in~\cite{LEbook-chap2}, we have
\begin{align}
& d \left(  \dec^{(n)} (B)(x), \dec^\tau (B) (x) \right)   =  d (F_0 \sqcap F_0^{'}, \, F \sqcap F') \notag\\
& \qquad  \le  \max\left\{  \frac{1}{\theta_\sqcap (F_0,F_0')},
 \frac{1}{\theta_\sqcap (F,F')} \right\} \,( d_{\tau}(F_0,F) + d_{\tau^\perp}(F_0',F') ) \label{eq0 cod}\,. 
\end{align}

Applying \eqref{speed conv cod} to $B^\ast$ we get:
\begin{equation}\label{eq1 cod}
d_\tau (F_0, F) = d_\tau  \left(  \mostexp^{(n)}_\tau(B^\ast)(x),  \mostexp^{(\infty)}_\tau(B^\ast)(x) \right) < \frac{\varkappa_n}{\varepsilon_n}\,,
\end{equation}
while applying \eqref{speed conv cod} to $B$ and using the fact the the orthogonal complement $\perp$ is an isometry, we get:
\begin{align}
d_{\tau^\perp} (F_0', F') &= d_{\tau^\perp}  \left(  \mostexp^{(n)}_\tau(B)(x)^\perp,  \, \mostexp^{(\infty)}_\tau(B)(x)^\perp \right) \notag \\
& = d_\tau  \left(  \mostexp^{(n)}_\tau(B)(x),  \, \mostexp^{(\infty)}_\tau(B)(x) \right) < \frac{\varkappa_n}{\varepsilon_n} \label{eq2 cod}\,.
\end{align}

By Lemma~\ref{lemma transv bound cod} we have
\begin{equation} \label{eq3 cod}
\theta_{\sqcap}  (F_0, F _0') = \theta_{\sqcap}  \left(  \mostexp^{(n)}_\tau(B^\ast)(x),  \, \mostexp^{(n)}_\tau(B)(x)^\perp \right) \ge \varepsilon_n\,,
\end{equation}
and by Proposition 3.15 in~\cite{LEbook-chap2} combined with \eqref{eq1 cod} and \eqref{eq2 cod} we have:
\begin{align} 
\theta_{\sqcap}(F,F') & \geq \theta_{\sqcap}(F_0,F_0')
- d_{\tau}(F,F_0)-d_{\tau^\perp}(F',F_0') \notag \\
& \ge \varepsilon_n - \frac{\varkappa_n}{\varepsilon_n}  - \frac{\varkappa_n}{\varepsilon_n}  \more \varepsilon_n\,.\label{eq4 cod} 
\end{align}

We conclude by combining  \eqref{eq0 cod}-\eqref{eq4 cod}. 
\end{proof}

\begin{remark}\normalfont
The proposition above shows in particular that the partially defined finite scale decompositions $ \dec^{(n)} (B)(x)$ converge for $\mu$-a.e. $x$ to the $\tau$-restriction $\dec^\tau (B) (x)$ of the Oseledets decomposition of $B$.
\end{remark}

\begin{proposition}[finite scale continuity] \label{prop finite scale cod}
There are constants $C_1 = C_1 (A) < \infty$ and $C_3 = C_3 (A) > 0$ such that for any cocycles $B_i \in \cocycles_m$ with $\dist (B_i, A) < \delta$ for $i=1, 2$, if $\dist (B_1, B_2) < e^{ - C_1 \, n}$, then for $x$ outside a set of measure $ < \mes_n$ and $n \ge \nzerobar$ we have:
\begin{equation}\label{eq finite scale cod}
d \left(  \dec^{(n)} (B_1)(x),  \,  \dec^{(n)} (B_2)(x)\right) < e^{- C_3 \, n}.
\end{equation}
\end{proposition}

\begin{proof}
Let $C_2 > \gapc (A)/2$. We apply item 4 of Remark~\ref{summary-estimates-dec}. There is $C_1 = C_1 (A)$ such that for any cocycles $B_i \in \cocycles_m$ with $\dist (B_i, A) < \delta$ for $i=1, 2$, there is a set of phases of measure $< \mes_n$ such that outside of that set,  
 \eqref{finite scale cod} holds for both $B_1, B_2$ and $B_1^\ast, B_2^\ast$. 
 
 Fix such a phase $x$, and to simplify notations, for $i=1, 2$ let
 $$F_i := \mostexp^{(n)}_\tau(B_i^\ast) (x), \quad F_i' := \mostexp^{(n)}_\tau(B_i)(x)^\perp\,,$$
 hence $ \dec^{(n)} (B_i)(x) = F_i \sqcap F_i'$.
 
 By Proposition 3.16 in~\cite{LEbook-chap2}, we have
\begin{align}
& d \left(  \dec^{(n)} (B_1)(x), \,  \dec^{(n)} (B_2)(x)\right)    = d (F_1 \sqcap F_1', \, F_2 \sqcap F_2') \notag\\
& \qquad   \le  \max\left\{  \frac{1}{\theta_\sqcap (F_1,F_1')},
 \frac{1}{\theta_\sqcap (F_2,F_2')} \right\} \,( d_{\tau}(F_1,F_2) + d_{\tau^\perp}(F_1',F_2') ) \label{eq7 cod}\,. 
\end{align}
 
Applying  \eqref{finite scale cod} to $B_1^\ast, B_2^\ast$ we get
\begin{equation}\label{eq 8 cod}
d_{\tau}(F_1,F_2)  = d_{\tau} \left(  \mostexp^{(n)} (B_1^\ast)(x), \,  \mostexp^{(n)} (B_2^\ast)(x)\right) < e^{-C_2 \, n},
\end{equation}
and applying  \eqref{finite scale cod} to $B_1, B_2$ we get
\begin{align}
d_{\tau^\perp}(F_1',F_2')  & = d_{\tau^\perp} \left(  \mostexp^{(n)} (B_1)(x)^\perp, \, \mostexp^{(n)} (B_2)(x)^\perp  \right) \notag \\
& = d_{\tau} \left(  \mostexp^{(n)} (B_1)(x), \,  \mostexp^{(n)} (B_2)(x)\right)  < e^{-C_2 \, n}. \label{eq 9 cod}
\end{align} 

By Lemma~\ref{lemma transv bound cod} we have, for $i=1, 2$:
\begin{equation} \label{eq10 cod}
\theta_{\sqcap}  (F_i, F _i') = \theta_{\sqcap}  \left(  \mostexp^{(n)}_\tau(B_i^\ast)(x),  \, \mostexp^{(n)}_\tau(B_i)(x)^\perp \right) \ge \varepsilon_n = e^{- n \, \gapc(A)/2}\,.
\end{equation}

Combining \eqref{eq7 cod}-\eqref{eq10 cod} we conclude:
$$d \left(  \dec^{(n)} (B_1)(x),  \,  \dec^{(n)} (B_2)(x)\right)  \less  e^{n \, \gapc(A)/2} \, e^{-C_2 \, n}  < e^{-C_3 \, n},$$
for an appropriate constant $C_3$, which proves the proposition.  
\end{proof}

We formulate the ACT for the Oseledets decomposition.

\begin{theorem} \label{thm abstract cod}
Let $A \in \cocycles_m$ be a cocycle with a $\tau$ gap pattern. Then locally near $A$, the map $$\cocycles_m \ni B \mapsto \dec^\tau (B) \in \Dec_\tau (X, \R^m)$$ is continuous with a modulus of continuity $\omega (h) :=  [ \mesf \ (c \, \log \frac{1}{h}) ]^{1/2}$ for some constant $c = c (A) > 0$ and for some deviation measure function  $\mesf = \mesf (A)$ from the corresponding set of LDT parameters.


In fact, a stronger pointwise estimate holds:
$$\mu  \ \{  x \in X \colon d ( \dec^\tau (B_1) (x), \  \dec^\tau (B_2) (x) ) > \dist (B_1, B_2)^\alpha \}  
 < \omega (\dist (B_1, B_2))\,,$$
for any $B_1, B_2 \in \cocycles_m$ in a neighborhood of $A$, and for  $\alpha = \alpha (A) >0$.

Moreover, the map $\cocycles_m \ni A \mapsto \dec (A) \in \Dec(X, \R^m)$ is continuous everywhere.
\end{theorem}

\begin{proof}
The first two assertions are derived from the speed of convergence in Proposition~\ref{prop speed conv cod} and the finite scale continuity in Proposition~\ref{prop finite scale cod} exactly the same way we derived the continuity of the most expanding direction in Theorem~\ref{thm:most-dir}. 

The third assertion is an immediate consequence of Proposition~\ref{topology-meas-decomp}. 
\end{proof}

\subsection*{Acknowledgments}

The first author was supported by 
Funda\c{c}\~{a}o  para a  Ci\^{e}ncia e a Tecnologia, 
UID/MAT/04561/2013.

The second author was supported by the Norwegian Research Council project no. 213638, "Discrete Models in Mathematical Analysis".

\bigskip

\bibliographystyle{amsplain} 

\def\cprime{$'$}
\providecommand{\bysame}{\leavevmode\hbox to3em{\hrulefill}\thinspace}
\providecommand{\MR}{\relax\ifhmode\unskip\space\fi MR }
\providecommand{\MRhref}[2]{%
  \href{http://www.ams.org/mathscinet-getitem?mr=#1}{#2}
}
\providecommand{\href}[2]{#2}

\end{document}